\documentclass[a4paper]{article}
\usepackage[margin=20mm]{geometry}
\usepackage[utf8]{inputenc}
\usepackage{graphics}
\usepackage{a4wide}
\usepackage{epsfig}
\usepackage{xspace}
\usepackage{url}
\usepackage{amsmath,amscd,amsthm}
\usepackage{amsfonts}
\usepackage{multirow,booktabs}
\usepackage{amssymb}
\usepackage{hyperref,tikz,tikz-cd}
\usepackage{caption}
\usepackage{accents}
\usepackage{subcaption}
\usepackage{pdfpages}
\usepackage[mathscr]{euscript}
\usepackage{mathrsfs}
\usepackage{listings}
 \usepackage{enumerate}
 \usepackage{cite}
 \usepackage{hyperref}
 \hypersetup{linkcolor=blue, colorlinks=true ,citecolor = red}
\usetikzlibrary{decorations.pathreplacing,decorations.markings}

\newtheorem{theorem}{Theorem}[section]
\newtheorem{lemma}[theorem]{Lemma}
\newtheorem{proposition}[theorem]{Proposition}
\newtheorem{corollary}[theorem]{Corollary}
\newtheorem{remark}[theorem]{Remark}
\newtheorem{definition}[theorem]{Definition}

\catcode`@=11 \@addtoreset{equation}{section}
\renewcommand\theequation{\thesection.\@arabic\c@equation}
\catcode`@=12

\makeatother
\usepackage[english]{babel}
\providecommand{\keywords}[1]
{
  \small	
  \textbf{\textit{Keywords---}} #1
}
\providecommand{\MSC}[1]
{
  \small	
  \textbf{\textit{Mathematics Subject Classification---}} #1
}
\usepackage{titlesec,authblk}

 \usepackage[hyperpageref]{backref}

\title{On the extreme value of Nehari manifold for nonlocal singular Schr{\"o}dinger-Kirchhoff equations in $\mathbb{R}^N$}
\date{}
\author{Deepak Kumar Mahanta$^{1}$, Tuhina Mukherjee$^{1}$ and Abhishek Sarkar$^{1,}$\thanks{Corresponding author}  \\
        \small $^{1}$ Department of Mathematics, Indian Institute of Technology Jodhpur, Rajasthan 342030, India \\
        }
\newcommand{\Addresses}{{
  \bigskip
  \footnotesize

  D.K.~Mahanta, \textit{E-mail address:} \texttt{mahanta.1@iitj.ac.in}

  T.~Mukherjee, \textit{E-mail address:} \texttt{tuhina@iitj.ac.in}

  A.~Sarkar, \textit{E-mail address:} \texttt{abhisheks@iitj.ac.in}

}}
\begin{document}
\maketitle \vspace{-1.8\baselineskip}
\begin{abstract}
This article investigates the existence, non-existence, and multiplicity of weak solutions for a parameter-dependent nonlocal Schr{\"o}dinger-Kirchhoff type problem on $\mathbb R^N$ involving singular non-linearity. By performing nuanced analysis based on Nehari submanifolds and fibre maps, our goal is to show the problem has at least two positive solutions even if $\lambda$ lies beyond the extremal parameter $\lambda_\ast$.
\end{abstract}
\\
\noindent\MSC{35J50, 35J75, 35J60, 35R11}
\section{Introduction}
Given $0<s<1$, $p\in(1,\infty)$ and $sp<N$, we consider the following fractional $p$-Laplacian singular problem of Schr{\"o}dinger-Kirchhoff type:
\begin{equation}\label{main problem}
  \begin{cases}
M\big( \|u\|^p\big)\bigg[\big(-\Delta_p\big)^s u+V(x)u^{p-1}\bigg]=\lambda\alpha(x)u^{-\delta}+\beta(x)u^{\gamma-1}~~\text{in}~~\mathbb{R}^N,\\
~~~~~~u>0~~\text{in}~~\mathbb{R}^N,~\displaystyle \int_{\mathbb{R}^N} V(x)u^p~\mathrm{d}x<\infty,~ u\in W^{s,p}(\mathbb{R}^N),
  \end{cases}\tag{$\mathcal{E}_\lambda$}
\end{equation}
 with 
 $$ \|u\|=\Bigg(\int_{\mathbb{R}^{N}}\int_{\mathbb{R}^{N}}\frac{|u(x)-u(y)|^p}{|x-y|^{N+sp}}~\mathrm{d}x\mathrm{d}y+\int_{\mathbb{R}^{N}}V(x)|u|^p~\mathrm{d}x\Bigg)^{\frac{1}{p}},$$
 where $N\geq 2$, $0<\delta<1$, $\lambda>0$ is a real parameter, $M:\mathbb{R}^+\to\mathbb{R}^+$ (where $\mathbb{R}^+:=(0,\infty)$) is a Kirchhoff function, $V:\mathbb{R}^N\to \mathbb{R}^+ $ is a potential function and $\big(-\Delta_p\big)^s$ is the fractional $p$-Laplacian operator which, up to a normalization constant, is defined as
 $$ \big(-\Delta_p\big)^s \phi(x)= 2\lim_{\epsilon\to 0^+}\int_{\mathbb{R}^N\setminus B_{\epsilon}(x)} \frac{|\phi(x)-\phi(y)|^{p-2}(\phi(x)-\phi(y))}{|x-y|^{N+sp}}~\mathrm{d}y,~~x\in \mathbb{R}^N, $$
 along any $\phi\in\mathcal{C}_0 ^{\infty}(\mathbb{R}^N)$, where $B_{\epsilon}(x)=\{y\in \mathbb{R}^N: |x-y|<\epsilon\}$.
More details on the fractional $p$-Laplacian operator can be found in \cite{goyal2015existence,franzina2014fractional} and the references therein. Throughout the paper, we assume the following hypotheses: 
\begin{itemize}
\item[(H1)] $0<\delta<1<p<\gamma<p^*_{s}-1$, where $p^*_{s}=\frac{Np}{N-sp}$ is the fractional Sobolev critical exponent.
\item[(H2)] $\alpha>0$ a.e. in $\mathbb{R}^N$ and there exists $\xi>1$ such that $ \alpha\in L^\xi(\mathbb{R}^N)\cap L^{\infty}(\mathbb{R}^N)$. Furthermore, there exists another constant $\tau>1$ such that $p<\tau<p^*_{s}$ satisfies $\frac{1}{\xi}+\frac{1-\delta}{\tau}=1$.
\item[(H3)] $\beta$ can be sign changing a.e. in $\mathbb{R}^N$ with $\beta^+=\max ~\{\beta,0\}\neq 0$ and $\beta\in L^{\infty}(\mathbb{R}^N)$.
\item[(H4)] If $\beta(x)>0$ in $\mathbb{R}^N$, then $\bigg(\frac{\alpha(x)}{\beta(x)}\bigg)^\frac{1}{\gamma+\delta-1}\notin  W^{s,p}_V(\mathbb{R}^N)$.
\item[(H5)] The Kirchhoff function  $M:\mathbb{R}^+\to\mathbb{R}^+$ is defined by $M(s)=as^m ~(a,m>0)$, where $p(m+1)<\gamma$, is a continuous and monotonically increasing function. Moreover, we define $$\widehat{M}(t)=\int_{0}^{t}M(s)~\mathrm{d}s,~\forall~t>0.$$
\end{itemize}
Furthermore, to avoid the lack of compactness of the embeddings of the solution space into Lebesgue space, we introduce the following conditions on $V:\mathbb{R}^N\to \mathbb{R}^+$:
\begin{itemize}
\item[(V1)] $V\in\mathcal{C}(\mathbb{R}^N)$ and there exists $V_0>0$ such that $\displaystyle\inf_{\mathbb{R}^N} V\geq V_0$.
\item[(V2)] There exists $h>0$ such that $\displaystyle\lim_{|y|\to\infty} ~\mathrm{meas}\{x\in B_h(y):V(x)\leq c\}=0 ,$ for all $c>0$, where $B_R(x)$ denotes any open ball in $\mathbb{R}^N$ centered at $x$ and of radius $R>0$. By $\mathrm{meas}(E)$, we denote the Lebesgue measure of $E \subset \mathbb{R}^N$.
\end{itemize}
\begin{remark}
 The condition $\mathrm{(V2)}$ is weaker than the coercivity assumption, that is, $V(x)\to\infty~\text{as}~|x|\to\infty$.   
\end{remark}

 The study of Kirchhoff-type problems, which Kirchhoff originally studied in \cite{kirchhoff1897vorlesungen}, has received a great deal of attention in past years due to its applicability in a wide range of models of physical and biological systems. The Kirchhoff function $M$ is often represented by the expression $M(t)=a+bt^\theta$ for $t\geq 0$ and $\theta>0$, where $a,b\geq 0$ and $a+b>0$. Consequently, $M$ is considered a degenerate Kirchhoff function if and only if $a=0$, while it is nondegenerate if $a>0$. However, we note that the degenerate situation in Kirchhoff's theory is much more fascinating and delicate than the non-degenerate case. In this direction, we mention \cite{ambrosio2022kirchhoff,chen2013multiple,pucci2019existence,pucci2015multiple} and references therein for interested readers.

A very well-established fact says that the Nehari submanifolds (defined in Section \ref{sec3}) $\mathcal{M}^\pm_\lambda$ are separated from $\mathcal{M}^0_\lambda$, in the sense that the boundaries of $\mathcal{M}^\pm_\lambda$ do not overlap with $\mathcal{M}^0_\lambda$, then minimizing the energy functional over
$\mathcal{M}^\pm_\lambda$, one can easily show the existence of at least two nonnegative solutions to the corresponding problem. The main difficulties arise when the boundaries of these submanifolds overlap with each other. To avoid these difficulties, we introduce the extremal threshold value $\lambda^\ast$ (see \eqref{eq3.19}), in the sense that when $\lambda<\lambda^\ast$, then $\mathcal{M}_\lambda$ is a $\mathcal{C}^1$-manifold while if $\lambda\geq\lambda^\ast$, then $\mathcal{M}^0_\lambda\neq \emptyset$ and $\mathcal{M}_\lambda$ is not a manifold. Consequently, whenever $\lambda\geq\lambda^\ast$, the standard minimizing techniques are of no help to find the local minimizers for the energy functional over $\mathcal{M}^\pm_\lambda$ and the situation becomes more complicated. We need some topological estimates for the Nehari sets to overcome these technical issues and obtain local minimizers for the energy functional. In that context, we refer the reader to  \cite{alves2022multiplicity,ilyasov2017extreme,ilyasov2018branches,quoirin2023local,quoirin2021nehari,silva2018extremal,de2020extreme} and references therein.

In the nonlocal framework, many authors have extensively studied elliptic problems with singular nonlinearity involving fractional $p$-Laplacian over bounded domains using the Nehari manifold technique and the fibre map analysis. For instance, we refer \cite{mukherjee2016dirichlet,wang2019existence,goyal2017multiplicity,goyal2016fractional} and references therein for further readings in this direction. Whereas we refer to \cite{wang2021uniqueness,xiang2020least,xiang2016multiplicity,fareh2023multiplicity,do2019nehari} for the results concerning Kirchhoff-type problems. It is worth mentioning that very few contributions are devoted to the study of nonlocal Kirchhoff problems with singular nonlinearity in bounded domains and the whole space $\mathbb{R}^N$. In this direction, we refer the reader to \cite{hsini2019multiplicity,fiscella2019nehari,wang2020combined}.

Motivated by the works mentioned above, our objective in this paper is to generalize the results of \cite{alves2022multiplicity}. To the best of the authors' knowledge, the nonlocal problems involving fractional $p$-Laplacian operator with singular nonlinearity and sign-changing weight function in $\mathbb{R}^N$ in the context of extremal parameters of Nehari manifold have not yet been studied. Our paper is a standard contribution to the existing literature. However, we combine the known techniques due to the appearance of degenerate Kirchhoff function, the singular term,  and whole space $\mathbb R^N$ in \eqref{main problem}. Each feature has its characteristics to enhance the novelty of the situation, and we try to enlist them now. The singular term in \eqref{main problem} finishes the possibility to differentiate the associated energy functional; the degenerate Kirchhoff function puts many barriers to obtaining strong convergence from weak convergence of minimizing sequences. To get over the trouble caused by the degenerate Kirchhoff function, we build an operator (see Lemma \ref{lemma2.4}) to the fixed weakly convergent sequence for establishing compactness. To handle the noncompactness caused due to $\mathbb R^N$, the assumptions $(V1)$ and $(V2)$ come to our rescue. Last but not least, we study the existence and nonexistence of solutions to \eqref{main problem} beyond the extremal value $\lambda_*$. Our efforts lie in combining these techniques efficiently to obtain the main results listed below.

Before we state our main results, we define the weak solution and the energy functional associated with \eqref{main problem}.
\begin{definition}[Weak Solution]
We say that $u\in W^{s,p}_{V}(\mathbb{R}^N)$ is a weak solution of \eqref{main problem} if $\alpha(x)u^{-\delta}v\in L^1(\mathbb{R}^N)$, $u>0$ for a.e. $x\in \mathbb{R}^N$ and it satisfies
$$M\big(\|u\|^p\big)\bigg[\int_{\mathbb{R}^{N}}\int_{\mathbb{R}^{N}}\frac{|u(x)-u(y)|^{p-2}(u(x)-u(y))(v(x)-v(y))}{|x-y|^{N+sp}}~\mathrm{d}x\mathrm{d}y+\int_{\mathbb{R}^{N}}V(x)u^{p-1}v~\mathrm{d}x\bigg]$$ $$
-\lambda \int_{\mathbb{R}^{N}} \alpha(x)u^{-\delta}v~\mathrm{d}x-\int_{\mathbb{R}^{N}}\beta(x)u^{\gamma-1}v~\mathrm{d}x=0,~\text{for all}~v\in W^{s,p}_{V}(\mathbb{R}^N). $$
\end{definition}
The energy functional $\Psi_\lambda:W^{s,p}_{V}(\mathbb{R}^N)\to\mathbb{R}$ associated with \eqref{main problem} is given by
\begin{equation}\label{eq2.11}
 \Psi_\lambda(u)=\frac{1}{p}\widehat{M}\big(\|u\|^p\big)-\frac{\lambda}{1-\delta}\int_{\mathbb{R}^{N}}\alpha(x)|u|^{1-\delta}~\mathrm{d}x-\frac{1}{\gamma}\int_{\mathbb{R}^{N}}\beta (x)|u|^\gamma~\mathrm{d}x.   
\end{equation}

It is easy to see that $\Psi_\lambda$ is well defined and continuous in $W^{s,p}_{V}(\mathbb{R}^N)$. One can notice that due to the presence of the singular term the energy functional  $\Psi_\lambda$  is not $\mathcal{C}^1$ in $W^{s,p}_{V}(\mathbb{R}^N)$ and also it is not bounded from below on $W^{s,p}_{V}(\mathbb{R}^N)$. Hence, we cannot apply the critical point theory to verify the existence of weak solutions for \eqref{main problem}. To overcome this, we work with a set of Nehari manifolds.

The following states our main results.
\begin{theorem}\label{thm2.6}
Assume that $(H1)-(H5)$ and $(V1)-(V2)$ are satisfied, then there exists $\lambda_*>0$ such that for all $\lambda\in$ $(0,\lambda_*+\epsilon)$, where $\epsilon>0$ is sufficiently small, the problem \eqref{main problem} has at least two positive solutions $u_\lambda\in \mathcal{M}^+_\lambda$ and $w_\lambda\in \mathcal{M}^-_\lambda$, respectively.
\end{theorem}
The remainder of the article is structured as follows. In Section \ref{sec2}, we have provided the mathematical framework for our analysis. In Section \ref{sec3}, we introduce the setup of the Nehari manifold and the study of the fibre map for \eqref{main problem} and also describe the extremal parameter $\lambda_\ast$. Finally, in Sections \ref{sec4}--\ref{sec6}, we have established the existence and multiplicity of solutions to \eqref{main problem} for $0<\lambda<\lambda_\ast$, $\lambda=\lambda_\ast$ and $\lambda>\lambda_\ast$, respectively, and we prove Theorem \ref{thm2.6}.

\section{Preliminaries}\label{sec2}
In this section, we introduce some primary results and properties of the fractional Sobolev spaces and then provide some necessary lemmas that will be needed in the proof of our main results.
Let $0<s<1<p<\infty$ and $sp<N$. The fractional Sobolev space $ W^{s,p}(\mathbb{R}^N)$ is defined by
$$  W^{s,p}(\mathbb{R}^N)=\bigg\{u\in L^p(\mathbb{R}^N): \int_{\mathbb{R}^{N}}\int_{\mathbb{R}^{N}}\frac{|u(x)-u(y)|^p}{|x-y|^{N+sp}}~\mathrm{d}x\mathrm{d}y<\infty\bigg\},$$
equipped with the norm
$$\|u\|_{ W^{s,p}(\mathbb{R}^N)}=\bigg(\|u\|^p_ {L^{p}(\mathbb{R}^N)}+[u]^p_{s,p}\bigg)^\frac{1}{p},$$
where $[u]_{s,p}$ denotes the Gagliardo seminorm, defined by $$[u]_{s,p}=\bigg(\int_{\mathbb{R}^{N}}\int_{\mathbb{R}^{N}}\frac{|u(x)-u(y)|^p}{|x-y|^{N+sp}}~\mathrm{d}x\mathrm{d}y\bigg)^\frac{1}{p}.$$ It is well-known that $\bigg(W^{s,p}(\mathbb{R}^N),\|\cdot\|_{ W^{s,p}(\mathbb{R}^N)} \bigg)$ is a uniformly convex Banach space (see\cite{pucci2015multiple}). The fractional critical exponent is defined by
$$ p^*_{s}=\begin{cases}
    \frac{Np}{N-sp},~\text{if}~sp<N;\\
    \infty~~~~,~\text{if}~sp\geq N.
\end{cases}$$
The space $ W^{s,p}(\mathbb{R}^N)$ is continuously embedded in $L^q(\mathbb{R}^N)$ for any $q\in[p,p^*_{s}]$ but is compactly embedded in $L^q_{loc}(\mathbb{R}^N)$ for any $q\in[p,p^*_{s})$ (see \cite{kim2018multiplicity}). For more details about the space $ W^{s,p}(\mathbb{R}^N)$, we refer to \cite{di2012hitchhiker} and the references therein.
Moreover, for $1\leq p<\infty$, the space $L^p_{V}(\mathbb{R}^N)$ is consisting of all real-valued measurable functions, with $V(x)|u|^p\in L^1(\mathbb{R}^N)$, and endowed with the norm 
$$\|u\|_{L^p_{V}(\mathbb{R}^N)}=\bigg( \int_{\mathbb{R}^N}V(x)|u|^p~\mathrm{d}x\bigg)^\frac{1}{p},~\forall~u\in L^p_{V}(\mathbb{R}^N).$$
The space $\big(L^p_{V}(\mathbb{R}^N), \|\cdot\|_{L^p_{V}(\mathbb{R}^N)}\big)$ is also a uniformly convex Banach  space thanks to $(V1)$ (see \cite{pucci2015multiple}).
Now we define weighted fractional Sobolev space $ W^{s,p}_{V}(\mathbb{R}^N)$, which is a subspace of $ W^{s,p}(\mathbb{R}^N)$ and is defined by
$$ W^{s,p}_{V}(\mathbb{R}^N)=\bigg\{u\in W^{s,p}(\mathbb{R}^N):\int_{\mathbb{R}^N}V(x)|u|^p~\mathrm{d}x<\infty\bigg\},$$
endowed with the norm
$$ \|u\|=\bigg(\|u\|^p_{L^p_{V}(\mathbb{R}^N)}+[u]^p_{s,p}\bigg)^\frac{1}{p}.$$
It is easy to see that the space $\big(W^{s,p}_{V}(\mathbb{R}^N),\|\cdot\|\big)$ is a unifomly convex Banach space (see \cite{pucci2015multiple}).

\begin{lemma}[see\cite{pucci2015multiple}]\label{lemma2.1}
  Let $(V1)$ holds. Then the embeddings $$W^{s,p}_{V}(\mathbb{R}^N)\hookrightarrow W^{s,p}(\mathbb{R}^N)\hookrightarrow L^\vartheta (\mathbb{R}^N)$$ are continuous for any $\vartheta\in [p,p^*_{s}]$. Consequently, we have  $\textit{min}~\{1,V_0\}\|u\|^p_{ W^{s,p}(\mathbb{R}^N)}\leq \|u\|^p$ for all $u\in W^{s,p}_{V}(\mathbb{R}^N)$.
  Also, when $\vartheta\in [1,p^*_{s})$, then the embedding $W^{s,p}_{V}(\mathbb{R}^N)\hookrightarrow L^\vartheta (B_R(0))$ is compact for any $R>0$.
\end{lemma}
\begin{lemma}[see\cite{pucci2015multiple}]\label{lemma2.2} 
Suppose $(V1)$ and $(V2)$ hold. Let $\theta\in [p,p^*_{s})$ and $\{v_k\}_{k\in \mathbb{N}}$ be a bounded sequence in $W^{s,p}_{V}(\mathbb{R}^N)$. Then there exists $v\in W^{s,p}_{V}(\mathbb{R}^N)\cap L^\theta (\mathbb{R}^N)$ such that up to a subsequence $v_k\to v$ in $L^\theta (\mathbb{R}^N)$ as $k\to\infty$.
\end{lemma}
\begin{lemma}[see \cite{liang2017multiplicity}]\label{lemma2.3}
Let $(V1)$ and $(V2)$ are satisfied.
Then the embedding $W^{s,p}_{V}(\mathbb{R}^N)\hookrightarrow L^\tau (\mathbb{R}^N)$ is compact for any $\tau\in (p,p^*_{s})$.
\end{lemma}
Let $\big(W^{s,p}_{V}(\mathbb{R}^N)\big)^*$ is the continuous dual of $W^{s,p}_{V}(\mathbb{R}^N)$ and 
 $\langle{\cdot,\cdot}\rangle$ denotes the duality pair between $\big(W^{s,p}_{V}(\mathbb{R}^N)\big)^*$ and $W^{s,p}_{V}(\mathbb{R}^N)$. Now we define the operator $L: W^{s,p}_{V}(\mathbb{R}^N)\to \big(W^{s,p}_{V}(\mathbb{R}^N)\big)^*$ by $$\langle{L(u),v}\rangle =M(\|u\|^p)\langle{B(u),v}\rangle,~\text{for any}~v\in W^{s,p}_{V}(\mathbb{R}^{N}),$$
 where $$ \langle{B(u),v}\rangle=\int_{\mathbb{R}^{N}}\int_{\mathbb{R}^{N}}\frac{|u(x)-u(y)|^{p-2}(u(x)-u(y))(v(x)-v(y))}{|x-y|^{N+sp}}~\mathrm{d}x\mathrm{d}y+\int_{\mathbb{R}^{N}}V(x)|u|^{p-2}uv~\mathrm{d}x. $$
\begin{lemma}[see \cite{akkoyunlu2019infinitely}]\label{lemma2.4}
 If $(H5)$ is satisfied, then
 \begin{itemize}
     \item[(i)] $L: W^{s,p}_{V}(\mathbb{R}^N)\to \big(W^{s,p}_{V}(\mathbb{R}^N)\big)^*$ is continuous, bounded and strictly monotone.
     \item[(ii)] $L$ is a map of type $(S_+)$, i.e., if $v_k\rightharpoonup v$ weakly in $W^{s,p}_{V}(\mathbb{R}^N)$ as $k\to\infty$ and $\displaystyle\limsup_{k\to\infty}\langle{L(v_k)-L(v),v_k-v}\rangle\leq 0,$ then $v_k\to v$ in $W^{s,p}_{V}(\mathbb{R}^N)$ as $k\to\infty$.
     \item[(iii)] $L: W^{s,p}_{V}(\mathbb{R}^N)\to \big(W^{s,p}_{V}(\mathbb{R}^N)\big)^*$ is a homeomorphism.
 \end{itemize}
\end{lemma}
\section{Review of the Nehari manifold set}\label{sec3}
This section provides technical results on the fibre maps and Nehari manifold set. Since we aim to study positive solutions to the problem \eqref{main problem}, therefore we restrict the energy functional $\Psi_\lambda$ to a cone of non-negative functions (say $\mathcal{K}$ ) of $W^{s,p}_{V}(\mathbb{R}^N)$ defined by

$$ \mathcal{K}=\big\{u\in W^{s,p}_{V}(\mathbb{R}^N)\setminus\{0\}:u\geq 0 \big\}. $$
For $u\in \mathcal{K} $, define the $\mathcal{C}^\infty$-fiber map $\mathcal{F}_{\lambda,u}: \mathbb{R}^+\to \mathbb{R}$ by $\mathcal{F}_{\lambda,u}(t)=\Psi_{\lambda}(tu)$ for all $t,\lambda>0$. We can easily derive that
$$\hspace{-4cm}\mathcal{F}_{\lambda,u}(t)=\frac{1}{p}\widehat{M}\big(t^p\|u\|^p\big)-\frac{\lambda}{1-\delta}t^{1-\delta}\int_{\mathbb{R}^{N}}\alpha(x)|u|^{1-\delta}~\mathrm{d}x-\frac{t^\gamma}{\gamma}\int_{\mathbb{R}^{N}}\beta (x)|u|^\gamma~\mathrm{d}x $$
$$\hspace{-1.7cm}=\frac{at^{p(m+1)}}{p(m+1)}\|u\|^{p(m+1)}-\frac{\lambda}{1-\delta}t^{1-\delta}\int_{\mathbb{R}^{N}}\alpha(x)|u|^{1-\delta}~\mathrm{d}x-\frac{t^\gamma}{\gamma}\int_{\mathbb{R}^{N}}\beta (x)|u|^\gamma~\mathrm{d}x, $$
$$\hspace{-3.2cm}\mathcal{F}'_{\lambda,u}(t)=t^{p-1} M\big(t^p\|u\|^p\big)\|u\|^p-\lambda t^{-\delta}\int_{\mathbb{R}^{N}}\alpha(x)|u|^{1-\delta}~\mathrm{d}x-t^{\gamma-1}\int_{\mathbb{R}^{N}}\beta (x)|u|^\gamma~\mathrm{d}x$$
$$\hspace{-1.8cm}=at^{p(m+1)-1}\|u\|^{p(m+1)}-\lambda t^{-\delta}\int_{\mathbb{R}^{N}}\alpha(x)|u|^{1-\delta}~\mathrm{d}x-t^{\gamma-1}\int_{\mathbb{R}^{N}}\beta (x)|u|^\gamma~\mathrm{d}x $$
and $$\hspace{-5.8cm}\mathcal{F}''_{\lambda,u}(t)=(p-1)t^{p-2} M\big(t^p\|u\|^p\big)\|u\|^p +pt^{2p-2}M'\big(t^p\|u\|^p\big)\|u\|^{2p}$$ $$\hspace{2.5cm}+\lambda\delta t^{-\delta-1}\int_{\mathbb{R}^{N}}\alpha(x)|u|^{1-\delta}~\mathrm{d}x-(\gamma-1)t^{\gamma-2}\int_{\mathbb{R}^{N}}\beta (x)|u|^\gamma~\mathrm{d}x$$
$$\hspace{1.5cm}=a(p(m+1)-1)t^{p(m+1)-2}\|u\|^{p(m+1)}+\lambda\delta t^{-\delta-1}\int_{\mathbb{R}^{N}}\alpha(x)|u|^{1-\delta}~\mathrm{d}x-(\gamma-1)t^{\gamma-2}\int_{\mathbb{R}^{N}}\beta (x)|u|^\gamma~\mathrm{d}x. $$
Define the Nehari manifold set as
$$ \mathcal{M}_\lambda=\big\{u\in\mathcal{K}:\mathcal{F}'_{\lambda,u}(1)=0\big\}.$$
It is easy to verify that $\mathcal{F}'_{\lambda,u}(t)=0$ if and only if $tu\in \mathcal{M}_\lambda$. In particular, $\mathcal{F}'_{\lambda,u}(1)=0$ if and only if $u\in \mathcal{M}_\lambda$. It follows that every weak solution of \eqref{main problem} always belongs to $\mathcal{M}_\lambda$. Therefore, it is natural to split $\mathcal{M}_\lambda$ into three disjoint sets corresponding to local minima, local maxima, and points of inflexion, that is,
$$\mathcal{M}^\pm_\lambda =\bigg\{u\in\mathcal{M}_\lambda:\mathcal{F}''_{\lambda,u}(1)\overset{>}{<} 0\bigg\}~\text{and}~\mathcal{M}^0_\lambda =\big\{u\in\mathcal{M}_\lambda:\mathcal{F}''_{\lambda,u}(1)=0\big\}. $$
\begin{theorem}\label{thm3.1} The following results hold:
\begin{itemize}
    \item [(i)] $\Psi_{\lambda}$ is weakly lower semicontinuous.
    \item[(ii)] $\Psi_{\lambda}$ is coercive and bounded from below on $\mathcal{M}_\lambda $. In particular, $\Psi_{\lambda}$ is coercive and bounded from below on $\mathcal{M}^{+}_\lambda$ and $\mathcal{M}^{-}_\lambda$ respectively.
\end{itemize}   
\end{theorem}
\begin{proof} $(i)$ Let $\{u_n\}_{n\in \mathbb{N}}\subset W^{s,p}_{V}(\mathbb{R}^N)$ be such that  $u_n\rightharpoonup u$ weakly in $W^{s,p}_{V}(\mathbb{R}^N)$ as $n\to\infty$. From Lemma \ref{lemma2.3}, we get $u_n\to u$  in $L^\tau(\mathbb{R}^N)$ and $L^\gamma(\mathbb{R}^N)$ respectively as $n\to\infty$, $u_n\to u$ a.e. in $\mathbb{R}^N$ and there exist two functions $g_\tau$ and $h_\gamma$  in $L^\tau(\mathbb{R}^N)$ and $L^\gamma(\mathbb{R}^N)$  such that $|u_n(x)|\leq g_\tau(x)$ a.e. in $\mathbb{R}^N$ and $|u_n(x)|\leq h_\gamma(x)$ a.e. in $\mathbb{R}^N$. This implies that $\big||u_n|^{1-\delta}-|u|^{1-\delta}\big|^{\frac{\tau}{1-\delta}}\to 0~\text{a.e. in} ~\mathbb{R}^N$
        and $\big||u_n|^{1-\delta}-|u|^{1-\delta}\big|^{\frac{\tau}{1-\delta}}\leq 2^{\frac{\tau}{1-\delta}} (g_\tau(x))^\tau\in L^1(\mathbb{R}^N)$.
        Applying the Lebesgue dominated convergence theorem, we obtain \begin{equation}\label{eq3.1}
            \int_{\mathbb{R}^N}\big||u_n|^{1-\delta}-|u|^{1-\delta}\big|^{\frac{\tau}{1-\delta}}~\mathrm{d}x\to 0~\text{as}~n\to\infty.
        \end{equation}
        Since $\big||u_n|^{1-\delta}-|u|^{1-\delta}\big|\in L^{\frac{\tau}{1-\delta}}(\mathbb{R}^N)$, therefore using $(H2)$, \eqref{eq3.1} and Holder's inequality, we deduce that
        $$\bigg|\int_{\mathbb{R}^N}\alpha(x)\big(|u_n|^{1-\delta}-|u|^{1-\delta}\big)~\mathrm{d}x\bigg|\leq \|\alpha\|_{L^\xi(\mathbb{R}^N)}\||u_n|^{1-\delta}-|u|^{1-\delta}\|_{L^{\frac{\tau}{1-\delta}}(\mathbb{R}^N)}\to 0~\text{as}~n\to\infty.$$
        It follows that 
        \begin{equation}\label{eq3.2}
            \lim_{n\to\infty}\int_{\mathbb{R}^N}\alpha(x)|u_n|^{1-\delta}~\mathrm{d}x=\int_{\mathbb{R}^N}\alpha(x)|u|^{1-\delta}~\mathrm{d}x.
        \end{equation}
       Similarly, we can prove that
        \begin{equation}\label{eq3.4}
\lim_{n\to\infty}\int_{\mathbb{R}^N}\beta(x)|u_n|^{\gamma}~\mathrm{d}x=\int_{\mathbb{R}^N}\beta(x)|u|^{\gamma}~\mathrm{d}x.   
        \end{equation}
     Since $\widehat{M}\big(\|u\|^p\big)$ is weakly lower semicontinuous, i.e., when $u_n\rightharpoonup u$ weakly in $W^{s,p}_{V}(\mathbb{R}^N)$ as $n\to\infty$, then
     \begin{equation}\label{eq3.5}
        \widehat{M}\big(\|u\|^p\big)\leq\displaystyle\liminf_{n\to\infty} \widehat{M}\big(\|u_n\|^p\big).
     \end{equation}
    It follows from \eqref{eq3.2}, \eqref{eq3.4} and \eqref{eq3.5} that $\Psi_\lambda(u)\leq\displaystyle\liminf_{n\to\infty}\Psi_\lambda(u_n)$. Hence, $\Psi_\lambda$ is a weakly lower semicontinuous.
    
\noindent $(ii)$ Let $u\in\mathcal{M}_\lambda$ such that $\|u\|\geq 1$. This yields $\mathcal{F}'_{\lambda,u}(1)=0$ and hence we obtain 
     \begin{equation}\label{eq3.6}
       \int_{\mathbb{R}^N}\beta(x)|u|^\gamma~\mathrm{d}x=a\|u\|^{p(m+1)}-\lambda \int_{\mathbb{R}^N}\alpha(x)|u|^{1-\delta}~\mathrm{d}x. 
     \end{equation}
     Further, using Holder's inequality and Lemma \ref{lemma2.3}, we deduce that there exists a constant $c>0$ such that
     \begin{equation}\label{eq3.7}
      \int_{\mathbb{R}^N}\alpha(x)|u|^{1-\delta}~\mathrm{d}x\leq c\|\alpha\|_{L^\xi(\mathbb{R}^N)}\|u\|^{1-\delta}.   
     \end{equation}
     Using \eqref{eq3.6} and \eqref{eq3.7} in \eqref{eq2.11}, we deduce that
     \begin{equation}\label{eq9.9}
       \Psi_\lambda(u)\geq a\bigg(\frac{1}{p(m+1)}-\frac{1}{\gamma}\bigg)\|u\|^{p(m+1)}-c\lambda\|\alpha\|_{L^\xi(\mathbb{R}^N)}\bigg(\frac{1}{1-\delta}-\frac{1}{\gamma}\bigg)\|u\|^{1-\delta}  
     \end{equation}
     $$\to\infty~\text{as}~\|u\|\to\infty,~\text{since} ~0<1-\delta<1<p(m+1)<\gamma.$$
      This shows that $\Psi_\lambda$ is coercive on $\mathcal{M}_\lambda$ . Let us define a function $f:\mathbb{R}^+\to \mathbb{R} $ by $$ f(t)=K_1 t^{p(m+1)}-K_2 t^{1-\delta},~\forall~ t>0,$$
      where $$ K_1=a\bigg(\frac{1}{p(m+1)}-\frac{1}{\gamma}\bigg)~\text{and}~K_2=c\lambda\|\alpha\|_{L^\xi(\mathbb{R}^N)}\bigg(\frac{1}{1-\delta}-\frac{1}{\gamma}\bigg).$$
      Thanks to elementary calculus, at the point $t(K_1,K_2,p,m,\delta)=\bigg(\frac{(1-\delta)K_2}{p(m+1)K_1}\bigg)^{\frac{1}{p(m+1)+\delta-1}},$ the function $f$ has a unique global minimum. This yields that $\Psi_\lambda(u)\geq f(t(K_1,K_2,p,m,\delta))$ and hence $\Psi_\lambda$ is bounded from below on $\mathcal{M}_\lambda$.This completes the proof.  
\end{proof}
\begin{lemma}\label{lem3.2}
Let $\lambda>0$, then the following results hold:
\begin{itemize}
    \item [(i)] $\sup\big\{\|u\|:u\in \mathcal{M}^+_\lambda\big\}<\infty$. Furthermore, $0>\inf\big\{\Psi_\lambda(u):u\in \mathcal{M}^+_\lambda\big\}>-\infty$.
    \item[(ii)] $\inf\big\{\|w\|:w\in \mathcal{M}^-_\lambda\big\}>0$ and $\sup\big\{\|w\|:w\in \mathcal{M}^-_\lambda,~\Psi_\lambda(w)\leq \jmath\big\}<\infty$ for any $\jmath>0$.
    Moreover, $ \inf\big\{\Psi_\lambda(w):w\in \mathcal{M}^-_\lambda\big\}>-\infty$.
\end{itemize}
    
\end{lemma}
\begin{proof} $(i)$ Let us assume that $u\in \mathcal{M}^+_\lambda$. From the definition of $\mathcal{M}^+_\lambda$, we have 
        \begin{equation}\label{eq3.10}
         \|u\|^{p(m+1)}<\frac{\lambda(\delta+\gamma-1)}{a(\gamma-p(m+1))}\int_{\mathbb{R}^{N}}\alpha(x)|u|^{1-\delta}~\mathrm{d}x.   
        \end{equation}
        
        Using \eqref{eq3.7} in \eqref{eq3.10}, we get
          $$\|u\|<\bigg[ \frac{c\lambda\|\alpha\|_{L^\xi(\mathbb{R}^N)}(\delta+\gamma-1)}{a(\gamma-p(m+1))}\bigg]^\frac{1}{p(m+1)+\delta-1}=\hat{C}~\text{(say)}.$$
          
        From the above inequality, we infer that $\sup\big\{\|u\|:u\in \mathcal{M}^+_\lambda\big\}<\infty$. Next, our aim is to show $0>\inf\big\{\Psi_\lambda(u):u\in \mathcal{M}^+_\lambda\big\}>-\infty$. To prove this, let us take a minimizing sequence $\{u_n\}_{n\in \mathbb{N}}$ in $\mathcal{M}^+_\lambda$ for $\Psi$. Since $\Psi_\lambda$ is coercive on $\mathcal{M}^+_\lambda$, therefore $\{u_n\}_{n\in N}$ must be bounded and hence up to a subsequence $u_n\rightharpoonup u$ in $W^{s,p}_{V}(\mathbb{R}^N)$ as $n\to\infty$. It follows from Theorem \ref{thm3.1}-$(i)$ that $-\infty<\Psi_\lambda(u)\leq\displaystyle\liminf_{n\to\infty}\Psi_\lambda(u_n)=\inf\big\{\Psi_\lambda(u):u\in \mathcal{M}^+_\lambda\big\}$. This shows that $\inf\big\{\Psi_\lambda(u):u\in \mathcal{M}^+_\lambda\big\}>-\infty$. From \eqref{eq2.11} and \eqref{eq3.10}, we obtain
        $$\Psi_\lambda(u)=a\bigg(\frac{1}{p(m+1)}-\frac{1}{\gamma}\bigg)\|u\|^{p(m+1)}-\lambda\bigg(\frac{1}{1-\delta}-\frac{1}{\gamma}\bigg) \int_{\mathbb{R}^N}\alpha(x)|u|^{1-\delta}~\mathrm{d}x~(\text{as}~ u\in\mathcal{M}^+_\lambda)$$
        $$<a\bigg(\frac{1}{p(m+1)}-\frac{1}{\gamma}\bigg)\|u\|^{p(m+1)}-\lambda\bigg(\frac{1}{1-\delta}-\frac{1}{\gamma}\bigg) \frac{a(\gamma-p(m+1))}{\lambda(\delta+\gamma-1)}\|u\|^{p(m+1)}$$
        $$\hspace{-3cm}=a\bigg(\frac{\gamma-p(m+1)}{\gamma}\bigg)\bigg(\frac{1}{p(m+1)}-\frac{1}{1-\delta}\bigg)\|u\|^{p(m+1)}<0. $$
        This infers that $\inf\big\{\Psi_\lambda(u):u\in \mathcal{M}^+_\lambda\big\}<0.$
        
\noindent $(ii)$ Let $w\in \mathcal{M}_\lambda^-$ and hence we have the following inequality
        \begin{equation}\label{eq3.14}
         \|w\|^{p(m+1)}<\frac{(\delta+\gamma-1)}{a(p(m+1)+\delta-1)}\int_{\mathbb{R}^{N}}\beta(x)|w|^{\gamma}~\mathrm{d}x.  
        \end{equation}
        By Lemma \ref{lemma2.3}, there exists $C>0$ such that
        \begin{equation}\label{eq3.15}
    \int_{\mathbb{R}^{N}}\beta(x)|w|^{\gamma}~\mathrm{d}x\leq C\|\beta\|_{L^\infty(\mathbb{R}^{N})}\|w\|^\gamma.   
        \end{equation}
        On solving \eqref{eq3.14} and \eqref{eq3.15}, we obtain
        \begin{equation}\label{eq3.16} 
          \|w\|>\bigg(\frac{a(p(m+1)+\delta-1)}{C\|\beta\|_{L^\infty(\mathbb{R}^{N})}(\delta+\gamma-1)}\bigg)^{\frac{1}{\gamma-p(m+1)}}=\hat{c}~\text{(say)}. 
        \end{equation} 
        The above inequality infer that $\inf\big\{\|w\|:w\in \mathcal{M}^-_\lambda\big\}>0$. Now from \eqref{eq9.9}, we get
        $$\jmath\geq\Psi_\lambda(w)\geq a\bigg(\frac{1}{p(m+1)}-\frac{1}{\gamma}\bigg)\|w\|^{p(m+1)}-c\lambda\|\alpha\|_{L^\xi(\mathbb{R}^N)}\bigg(\frac{1}{1-\delta}-\frac{1}{\gamma}\bigg)\|w\|^{1-\delta}.$$
       Since $1-\delta<1<p(m+1)$, we have $\sup\big\{\|w\|:w\in \mathcal{M}^-_\lambda,~\Psi_\lambda(w)\leq \jmath\big\}<\infty$. The proof of $\inf\big\{\Psi_\lambda(w):w\in \mathcal{M}^-_\lambda\big\}>-\infty$ is similar to the proof of $\inf\big\{\Psi_\lambda(u):u\in \mathcal{M}^+_\lambda\big\}>-\infty $ and hence we omit. This completes the proof. 
\end{proof}
\begin{corollary}\label{cor3.3}
      If $u\in\mathcal{M}^0_\lambda$, then  $\displaystyle\int_{\mathbb{R}^{N}}\beta (x)|u|^\gamma~\mathrm{d}x>0$ and $\displaystyle\int_{\mathbb{R}^{N}}\alpha (x)|u|^{1-\delta}~\mathrm{d}x>0 $. Moreover, the set $\mathcal{M}^0_\lambda$ is bounded with respect to the norm of $W^{s,p}_V(\mathbb{R}^N)$, for each $\lambda>0$.     
\end{corollary}
\begin{proof}
    The proof is obtained from the definition of $\mathcal{M}^0_\lambda$, Lemma \ref{lemma2.3} and Holder's inequality.
\end{proof}
\subsection{Descriptive analysis of the extremal parameter}\label{subsec3}
 Define $\mathcal{G}^+=\bigg\{u\in\mathcal{K}:~\displaystyle\int_{\mathbb{R}^{N}}\beta (x)|u|^\gamma~\mathrm{d}x>0\bigg\}$. Now, we characterize the set $\mathcal{M}^0_\lambda$ on the basis of the set $\mathcal{G}^+$. For this, let $u\in\mathcal{G}^+$ such that $tu\in\mathcal{M}^0_\lambda$. Hence, we have the system of equations $\mathcal{F}'_{\lambda,u}(t)=0$ and $\mathcal{F}''_{\lambda,u}(t)=0$. By solving these two equations, we can get a unique pair $(t(u),\lambda(u))$ as
\begin{equation}\label{eq3.18}
\begin{cases}
  t(u)=\bigg[ a\bigg(\frac{p(m+1)+\delta-1}{\gamma+p-1}\bigg)\bigg]^{\frac{1}{\gamma-p(m+1)}}\Bigg[\cfrac{\|u\|^{p(m+1)}}{\displaystyle\int_{\mathbb{R}^{N}}\beta (x)|u|^\gamma~\mathrm{d}x}\Bigg]^{\frac{1}{\gamma-p(m+1)}};\\ \\
  \lambda(u)=C(a,\gamma,\delta,m,p)\cfrac{\Big[\|u\|^{p(m+1)}\Big]^{\frac{\gamma+\delta-1}{\gamma-p(m+1)}}}{\bigg[\displaystyle\int_{\mathbb{R}^{N}}\alpha (x)|u|^{1-\delta}~\mathrm{d}x\bigg]\bigg[\displaystyle\int_{\mathbb{R}^{N}}\beta (x)|u|^\gamma~\mathrm{d}x\bigg]^{\frac{p(m+1)+\delta-1}{\gamma-p(m+1)}}},
\end{cases}    
\end{equation}
where $$C(a,\gamma,\delta,m,p)=a\bigg[\frac{\gamma-p(m+1)}{\gamma+\delta-1}\bigg] \bigg[ a\bigg(\frac{p(m+1)+\delta-1}{\gamma+p-1}\bigg)\bigg]^{\frac{p(m+1)+\delta-1}{\gamma-p(m+1)}}>0.$$
Due to Y.II'yasov (see\cite{ilyasov2017extreme,ilyasov2018branches}), we define the extremal parameter $\lambda_*$ by 
\begin{equation}\label{eq3.19}
  \lambda_* =\inf_{u\in\mathcal{G}^+}\lambda(u).  
\end{equation}
\begin{proposition}\label{prop3.1}
For any $u\in\mathcal{K}$, we have the following results:
\begin{itemize}
    \item [(I)] If $\lambda>0$ and $\displaystyle\int_{\mathbb{R}^{N}}\beta (x)|u|^\gamma~\mathrm{d}x\leq 0$, then $\mathcal{F}_{\lambda,u}$ has a unique critical point $t^+_\lambda(u)\in (0,\infty)$ such that $t^+_\lambda(u)u\in \mathcal{M}^+_\lambda$.
    \item[(II)] If $\displaystyle\int_{\mathbb{R}^{N}}\beta (x)|u|^\gamma~\mathrm{d}x> 0$, then there are three possibilities:
    \begin{itemize}
        \item [(a)] if $0<\lambda<\lambda(u)$, then $\mathcal{F}_{\lambda,u}$ has two unique critical points $0<t^+_\lambda(u)<t^-_\lambda(u)$ such that $t^+_\lambda(u)u\in \mathcal{M}^+_\lambda$ and $t^-_\lambda(u)u\in \mathcal{M}^-_\lambda$. Moreover, $\mathcal{F}_{\lambda,u}$ is decreasing over the intervals $(0,t^+_\lambda(u))$, $(t^-_\lambda(u),\infty)$ and increasing over the interval $(t^+_\lambda(u),t^-_\lambda(u))$.
        \item[(b)] if $\lambda=\lambda(u)$, then $\mathcal{F}_{\lambda,u}$ has a unique critical point $t^0_\lambda(u)>0$, which is a saddle point and $t^0_\lambda(u)u\in \mathcal{M}^0_\lambda$. Furthermore, $t^0_\lambda(u)=t(u)$ and $0<t^+_\lambda(u)<t(u)<t^-_\lambda(u)$ for all $\lambda\in(0,\lambda(u))$ and $\mathcal{F}_{\lambda,u}(t)$ decreases for $t>0$.
        \item[(c)] if $\lambda>\lambda(u)$, then $\mathcal{F}_{\lambda,u}$ is decreasing for $t>0$ and has no critical points.
    \end{itemize}
\end{itemize}
\end{proposition}
\begin{proof} 
Let us define the function $g_u:\mathbb{R}^+\to\mathbb{R} $ given by
\begin{equation}\label{eq3.97}
g_u(t)= at^{p(m+1)+\delta-1}\|u\|^{p(m+1)}-t^{\gamma+\delta-1}\int_{\mathbb{R}^{N}}\beta (x)|u|^\gamma~\mathrm{d}x~\text{for all}~t>0 .    
\end{equation} 
It is easy to see that for all $t>0$, $tu\in\mathcal{M}_\lambda$ if and only if $g_u(t)=\lambda\int_{\mathbb{R}^{N}}\alpha(x)|u|^{1-\delta}~\mathrm{d}x$. Hence we deduce that $t^{2-\delta}g'_u(t)=\mathcal{F}''_{\lambda,tu}(1)$ holds true whenever $tu\in\mathcal{M}_\lambda$. 

\noindent $(I)$ Let $\lambda>0$ and $\int_{\mathbb{R}^{N}}\beta (x)|u|^\gamma~\mathrm{d}x\leq 0$, then 
        $$t^{2-\delta}g'_u(t)=\mathcal{F}''_{\lambda,tu}(1)=a(p(m+1)-1)t^{p(m+1)}\|u\|^{p(m+1)}-(\gamma-1)t^\gamma\int_{\mathbb{R}^{N}}\beta (x)|u|^\gamma~\mathrm{d}x$$ $$+\lambda\delta t^{1-\delta}\int_{\mathbb{R}^{N}}\alpha(x)|u|^{1-\delta}~\mathrm{d}x>0. $$
        This shows that $g'_u(t)>0$, i.e., $g_u(t)$ is strictly increasing for $t>0$. Thus, there exists a unique positive real number $t^+_\lambda(u)\in (0,\infty)$ such that $t^+_\lambda(u)u\in \mathcal{M}^+_\lambda$.      
        
\noindent $(II)(a)$ By \eqref{eq3.97}, one can easily verify that $t=t(u)$ (see \eqref{eq3.18}) is a unique critical point for $g_u(t)$. In addition, we obtain $g'_u(t)\big|_{t=t(u)}=0$ and $g''_u(t)\big|_{t=t(u)}<0$. This implies that $g_u(t)$ has a global maximum at $t=t(u)$ and hence $g_u(t)$ is strictly increasing in the interval $(0,t(u))$ and strictly decreasing in the interval $(t(u),\infty)$. Since $t^{2-\delta}g'_u(t)=\mathcal{F}''_{\lambda,tu}(1)$, therefore $\mathcal{F}''_{\lambda,tu}(1)>0$ for all $t\in(0,t(u))$ and $\mathcal{F}''_{\lambda,tu}(1)<0$ for all $t\in(t(u),\infty)$. Choosing $0<t^+_\lambda(u)<t(u)<t^-_\lambda(u)$, it is clear from the strictly monotonicity of the function $g_u(t)$ in the intervals $(0,t(u))$ and $(t(u),\infty)$ that $t^+_\lambda(u)$ and $t^+_\lambda(u)$ are unique positive real numbers  such that $t^+_\lambda(u)u\in \mathcal{M}^+_\lambda$ and $t^-_\lambda(u)u\in \mathcal{M}^-_\lambda$. Thus, $\mathcal{F}''_{\lambda,t^+_\lambda(u)u}(1)>0$ and $\mathcal{F}''_{\lambda,t^-_\lambda(u)u}(1)<0$. This shows that $\mathcal{F}_{\lambda,u}$ decreases over the intervals $(0,t^+_\lambda(u))$, $(t^-_\lambda(u),\infty)$ and increases over the interval $(t^+_\lambda(u),t^-_\lambda(u))$.

           \noindent \noindent $(II)(b)$ Assume $\lambda=\lambda(u)$, then we get a unique point $t=t(u)$ such that $g'_u(t(u))=0$. This yields $\mathcal{F}''_{\lambda(u),u}(t(u))=0$ and hence $t(u)u\in\mathcal{M}^0_{\lambda(u)}$ with $t^+_\lambda(u)<t(u)<t^-_\lambda(u)$. It is easy to see $\mathcal{F}'_{\lambda(u),u}(t)\leq 0$ for all $t>0$ \big( as $\mathcal{F}'_{\lambda(u),u}(t(u))=0$\big). Hence $\mathcal{F}_{\lambda(u),u}$ always decreases for $t>0$ with $t=t(u)$ as a saddle point.
           
\noindent \noindent $(II)(c)$ Let $\lambda(u)<\lambda$, then using $(b)$ we get $\mathcal{F}'_{\lambda,u}(t)<\mathcal{F}'_{\lambda(u),u}(t)\leq 0$ for all $t>0$. This concludes that $\mathcal{F}_{\lambda,u}$ is decreasing for $t>0$ and has no critical points.
\end{proof}
\begin{lemma}\label{lem3.5}
  The function $\lambda(u)$ defined as in \eqref{eq3.18} is continuous and $0$-homogeneous. Furthermore, the extremal parameter $\lambda_*>0$ and there exists $u\in\mathcal{G}^+$ such that $\lambda_*=\lambda(u)$.
\end{lemma}
\begin{proof}
   It is obvious that $\lambda(u)$ is continuous. Let us choose $t>0$, then $\lambda(tu)=\lambda(u)$, that is, the function $\lambda(u)$ is 0-homogeneous. Since $\lambda(u)$ is $0$-homogeneous, we can restrict $\lambda(u)$ to the set $\mathcal{G}^+\cap S$, where $S=\big\{u\in W^{s,p}_{V}(\mathbb{R}^N):\|u\|=1\big\}$. Now from \eqref{eq3.7} and \eqref{eq3.15}, we obtain
    $$\lambda_*=\inf_{\mathcal{G}^+\cap S} \lambda(u)\geq \frac{C(a,\gamma,\delta,m,p)}{c\|\alpha\|_{L^\xi(\mathbb{R}^N)}\big(C\|\beta\|_{L^\infty(\mathbb{R}^{N})}\big)^{\frac{p(m+1)+\delta-1}{\gamma-p(m+1)}}}>0.$$
    Next, choose a sequence $\{u_n\}_{n\in \mathbb{N}}$ in $\mathcal{G}^+\cap S$ such that $\lambda(u_n)\to \lambda_*$ as $n\to\infty$. Since $\|u_n\|=1,~\forall~n\in \mathbb{N}$, therefore without loss of generality we can assume $u_n\rightharpoonup u$ weakly in $W^{s,p}_{V}(\mathbb{R}^N)$ as $n\to\infty$ and hence $u_n\to u$ in $L^\tau(\mathbb{R}^N)$ and $L^\gamma(\mathbb{R}^N)$ respectively as $n\to\infty$. Notice that $u\neq 0$, otherwise $\lambda(u_n)\to \infty$ as $n\to\infty$. Let us choose $\frac{u}{\|u\|}\in \mathcal{G}^+\cap S$. Now we claim that $u_n\to u$ in $W^{s,p}_{V}(\mathbb{R}^N)$ as $n\to\infty$. If not, then by using the 0-homogeneity property of $\lambda $ and weakly lower semicontinuity of the norm, we obtain $$\lambda\bigg(\frac{u}{\|u\|}\bigg)=\lambda(u)<\liminf_{n\to\infty}\lambda(u_n)=\lambda_*, $$
    which is a contradiction. Thus, we must have $u_n\to u$ in $W^{s,p}_{V}(\mathbb{R}^N)$ as $n\to\infty$ and therefore $u\in \mathcal{G}^+\cap S$. The continuity of $\lambda$ implies $\lambda(u)=\lambda_*$. This completes the proof.    
\end{proof}
\begin{corollary}\label{cor3.6}
The following results hold:
\begin{itemize}
    \item [(a)] The sets $\mathcal{M}^+_\lambda$ and $\mathcal{M}^-_\lambda$ are non empty for $\lambda>0$.
    \item[(b)] $\mathcal{M}^0_\lambda=\emptyset$ for $\lambda\in(0,\lambda_*)$ and $\mathcal{M}^0_\lambda\neq\emptyset$ for $\lambda\in[\lambda_*,\infty)$.
    \item[(c)] $\mathcal{M}^0_{\lambda_\ast}=\bigg\{u\in \mathcal{M}_{\lambda_\ast}:~\displaystyle\int_{\mathbb{R}^{N}}\beta (x)|u|^\gamma~\mathrm{d}x>0~\text{and}~\lambda(u)=\lambda_*\bigg\}$.
\end{itemize}
\end{corollary}
\begin{proof}
    The proof is a direct consequence of Proposition \ref{prop3.1} and Lemma \ref{lem3.5}.
\end{proof}

Let $\lambda>0$, the we define the sets
$$ \hat{\mathcal{M}}_\lambda=\big\{u\in\mathcal{K}:~u\in\mathcal{G}^+~\text{and}~\lambda<\lambda(u)\big\} ~\text{and}~ \hat{\mathcal{M}}^+_\lambda=\bigg\{u\in\mathcal{K}:~\int_{\mathbb{R}^{N}}\beta (x)|u|^\gamma~\mathrm{d}x\leq 0\bigg\}.$$
Further, for $\lambda>0$, we define $\mathcal{I}^+_{\lambda}:\hat{\mathcal{M}}_\lambda\cup\hat{\mathcal{M}}^+_\lambda\to\mathbb{R}$ and $\mathcal{I}^-_{\lambda}:\hat{\mathcal{M}}_\lambda\to\mathbb{R}$ by 
$$\mathcal{I}^+_{\lambda}(u)=\Psi_\lambda(t^+_\lambda(u)u)~\text{and}~\mathcal{I}^-_{\lambda}(u)=\Psi_\lambda(t^-_\lambda(u)u).$$
Consider the following constrained minimization problems 
$$\Upsilon^+_\lambda=\inf\big\{\mathcal{I}^+_{\lambda}(u):u\in\mathcal{M}^+_\lambda\big\}~\text{and}~\Upsilon^-_\lambda=\inf\big\{\mathcal{I}^-_{\lambda}(u):u\in \mathcal{M}^+_\lambda\big\} .$$
\begin{lemma}\label{lem3.7}
   Let $u\in\mathcal{K}$ and $I$ be an open interval in $\mathbb{R}$ such that $t^{\pm}_{\lambda}(u)$ are well defined for all $\lambda\in I$. Then the following results hold:
   \begin{itemize}
       \item [(i)] the functions $I\ni\lambda\mapsto  t^{\pm}_{\lambda}(u)$ are $\mathcal{C}^\infty$. Furthermore, $I\ni\lambda\mapsto  t^{+}_{\lambda}(u)$ is strictly increasing, whereas $I\ni\lambda\mapsto  t^{-}_{\lambda}(u)$ is strictly decreasing.
       \item[(ii)] the functions $I\ni\lambda\mapsto  \mathcal{I}^{\pm}_{\lambda}(u)$ are $\mathcal{C}^\infty$ and strictly decreasing.
   \end{itemize}
   Consequently, the above results hold for $I=(0,\lambda_\ast)$ as well as $I=(\lambda_\ast,\lambda_\ast+\epsilon)$, where $\epsilon>0$ is small enough.
\end{lemma}
\begin{proof}
 The proof is similar to [\cite{alves2022multiplicity}, Lemma 2.7], and we omit it.
\end{proof}
\section{Existence of solutions for \texorpdfstring{$\lambda \in (0,\lambda_\ast)$}{Lg}}\label{sec4}
In this section, we show the existence of at least two positive solutions to \eqref{main problem} for $0<\lambda<\lambda_\ast$. 
\begin{lemma}\label{lem4.1}
Let $0<\lambda<\lambda_\ast$. Then there exists $u_\lambda\in\mathcal{M}^+_\lambda$ such that $\Psi_\lambda(u_\lambda)=\mathcal{I}^+_\lambda(u_\lambda)=\Upsilon^+_\lambda$.    
\end{lemma}
\begin{proof}
    Since the functional $\Psi_\lambda$ is bounded from below on $\mathcal{M}_\lambda$ and so on $\mathcal{M}^+_\lambda$. So, there exists a minimizing sequence $\{u_n\}_{n\in \mathbb{N}}\subset \mathcal{M}^+_\lambda$ for $\Psi_\lambda$, i.e.,    
$$\lim_{n\to\infty}\Psi_\lambda(u_n)=\lim_{n\to\infty}\mathcal{I}^+_\lambda(u_n)=\Upsilon^+_\lambda=\inf\big\{\Psi_\lambda(u):u\in \mathcal{M}^+_\lambda\big\}.$$    
    The coercivity of $\Psi_\lambda$ (see Theorem \ref{thm3.1}) implies that $\{u_n\}_{n\in \mathbb{N}}$ is bounded in $W^{s,p}_{V}(\mathbb{R}^N)$. Hence, up to a subsequence $u_n\rightharpoonup u_\lambda$ weakly in $W^{s,p}_{V}(\mathbb{R}^N)$ as $n\to\infty$. By Lemma \ref{lemma2.3}, we have $u_n\to u_\lambda$ in $L^\tau(\mathbb{R}^N)$ and $L^\gamma(\mathbb{R}^N)$ respectively as $n\to\infty$, $u_n\to u_\lambda$ a.e. in $\mathbb{R}^N$. Now similar to the proof of \eqref{eq3.2},  \eqref{eq3.4} and \eqref{eq3.5}, we obtain 
\begin{equation}\label{eq4.2}
\begin{cases}    \displaystyle\lim_{n\to\infty}\int_{\mathbb{R}^N}\alpha(x)|u_n|^{1-\delta}~\mathrm{d}x=\int_{\mathbb{R}^N}\alpha(x)|u_\lambda|^{1-\delta}~\mathrm{d}x,\\ 
\displaystyle\lim_{n\to\infty}\int_{\mathbb{R}^N}\beta(x)|u_n|^{\gamma}~\mathrm{d}x=\int_{\mathbb{R}^N}\beta(x)|u_\lambda|^{\gamma}~\mathrm{d}x,\\
\widehat{M}\big(\|u_\lambda\|^p\big)\leq\displaystyle\liminf_{n\to\infty} \widehat{M}\big(\|u_n\|^p\big)~\text{and}~u_\lambda\geq 0.    
\end{cases}
\end{equation}
    By using \eqref{eq4.2}, we get 
    \begin{equation}\label{eq4.3}
    \lim_{n\to\infty}\Psi_\lambda(u_n)\geq \Psi_\lambda(u_\lambda).
    \end{equation}
    Now we have to show $u_\lambda>0$, that is, $u_\lambda\neq 0$. Indeed, if not, then from \eqref{eq4.3} we obtain $\hat{\mathcal{I}}^+_\lambda\geq 0$, which is a contradiction because $\hat{\mathcal{I}}^+_\lambda<0$ (see Lemma \ref{lem3.2}). This yields $u_\lambda\neq0$ and $u_\lambda\in\mathcal{K}$. If $\displaystyle\int_{\mathbb{R}^N}\beta(x)|u_\lambda|^{\gamma}~\mathrm{d}x>0$ or $\displaystyle\int_{\mathbb{R}^N}\beta(x)|u_\lambda|^{\gamma}~\mathrm{d}x\leq 0$, then by Proposition \ref{prop3.1} there exists a unique positive real number $t^+_\lambda(u_\lambda)$ such that $t^+_\lambda(u_\lambda)u\in\mathcal{M}^+_\lambda$, i.e., $\mathcal{F}'_{\lambda,u_\lambda}(t^+_\lambda(u_\lambda))=0$.  The next step is to prove $u_\lambda\in\mathcal{M}^+_\lambda$. For this, we claim that $u_n\to u_\lambda$ in $W^{s,p}_{V}(\mathbb{R}^N)$ as $n\to\infty$. Indeed, if not, then $u_n\not\to u_\lambda$ in $W^{s,p}_{V}(\mathbb{R}^N)$ as $n\to\infty$. Therefore, we have $\|u_\lambda\|^p<\displaystyle\liminf_{n\to\infty}\|u_n\|^p$ and $M\big(\|u_\lambda\|^p\big)<\displaystyle\liminf_{n\to\infty} M\big(\|u_n\|^p\big)$. From \eqref{eq4.2}, we obtain 
       $$\liminf_{n\to\infty}\mathcal{F}'_{\lambda,u_n}(t^+_\lambda(u_\lambda)) >\mathcal{F}'_{\lambda,u_\lambda}(t^+_\lambda(u_\lambda))=0.$$    
    It follows that $\mathcal{F}'_{\lambda,u_n}(t^+_\lambda(u_\lambda))>0$ for $n$ large enough. Since $u_n\in \mathcal{M}^+_\lambda$ for all $n\in N$, therefore 1 is the point of minimum for $\mathcal{F}_{\lambda,u_n}$, $\mathcal{F}'_{\lambda,u_n}(t)<0$ for all $t\in(0,1)$ and $\mathcal{F}'_{\lambda,u_n}(t)>0$ for all $t>1$. This shows that $t^+_\lambda(u_\lambda)>1$. Using the fact that $\mathcal{F}_{\lambda,u_\lambda}$ is strictly decreasing in $(0,t^+_\lambda(u_\lambda))$ (as $t^+_\lambda(u_\lambda)u_\lambda\in\mathcal{M}^+_\lambda$) and \eqref{eq4.3}, we get
    $$ \Psi_\lambda(t^+_\lambda(u_\lambda)u_\lambda)=\mathcal{F}_{\lambda,u_\lambda}(t^+_\lambda(u_\lambda))<\mathcal{F}_{\lambda,u_\lambda}(1)=\Psi_\lambda(u_\lambda)\leq \lim_{n\to\infty} \Psi_\lambda(u_n)=\Upsilon^+_\lambda,$$
    which is absurd because $t^+_\lambda(u_\lambda)u_\lambda\in\mathcal{M}^+_\lambda$. Thus, $u_n\to u_\lambda$ in $W^{s,p}_{V}(\mathbb{R}^N)$ as $n\to\infty$. This implies that $\mathcal{F}'_{\lambda,u_\lambda}(1)=0$ and $\mathcal{F}''_{\lambda,u_\lambda}(1)\geq 0$. By Corollary \ref{cor3.6}, $\mathcal{M}^0_\lambda =\emptyset$ for $0<\lambda<\lambda_\ast$, therefore $u_\lambda\in\mathcal{M}^+_\lambda$ and $\Psi_\lambda(u_\lambda)=\mathcal{I}^+_\lambda(u_\lambda)=\Upsilon^+_\lambda$. This completes the proof.
\end{proof}
\begin{lemma}\label{lem4.2}
Let $0<\lambda<\lambda_\ast$. Then there exists $w_\lambda\in\mathcal{M}^-_\lambda$ such that $\Psi_\lambda(w_\lambda)=\mathcal{I}^-_\lambda(w_\lambda)=\Upsilon^-_\lambda$.   
\end{lemma}
\begin{proof}
   The functional $\Psi_\lambda$ is bounded from below on $\mathcal{M}_\lambda$ and therefore on $\mathcal{M}^+_\lambda$. As a result, there is a minimizing sequence $\{u_n\}_{n\in \mathbb{N}}\subset \mathcal{M}^+_\lambda$ for $\Psi_\lambda$, i.e.,     
$$\lim_{n\to\infty}\Psi_\lambda(u_n)=\lim_{n\to\infty}\mathcal{I}^-_\lambda(u_n)=\Upsilon^-_\lambda=\inf\big\{\Psi_\lambda(w):w\in \mathcal{M}^-_\lambda\big\}.$$   
    Due to the coercivity of $\Psi_\lambda$ (see Theorem \ref{thm3.1}), the sequence $\{u_n\}_{n\in \mathbb{N}}$ must be bounded in $W^{s,p}_{V}(\mathbb{R}^N)$ and hence we can assume $u_n\rightharpoonup w_\lambda$ weakly in $W^{s,p}_{V}(\mathbb{R}^N)$ as $n\to\infty$. Now replacing $u_\lambda$ with $w_\lambda$ in \eqref{eq4.2}, one can observe that \eqref{eq4.2} still holds. 
    We claim that $w_\lambda\neq 0$. Indeed, if $w_\lambda=0$, then by Lemma \ref{lem3.2} (see \eqref{eq3.14}) and \eqref{eq4.2}, we get 
   $$ 0=\|w_\lambda\|^{p(m+1)}<\liminf_{n\to\infty} \|u_n\|^{p(m+1)}\leq \frac{(\delta+\gamma-1)}{a(p(m+1)+\delta-1)}\liminf_{n\to\infty}\int_{\mathbb{R}^{N}}\beta(x)|u_n|^{\gamma}~\mathrm{d}x$$ $$=\frac{(\delta+\gamma-1)}{a(p(m+1)+\delta-1)}\int_{\mathbb{R}^{N}}\beta(x)|w_\lambda|^{\gamma}~\mathrm{d}x=0, $$
   a contradiction and hence the claim. Repeating the above steps, one can easily deduce that $\displaystyle\int_{\mathbb{R}^{N}}\beta(x)|w_\lambda|^{\gamma}~\mathrm{d}x>0$ and $w_\lambda\in\mathcal{K}$. Now by Proposition \ref{prop3.1}, there exists $t^-_{\lambda}(w_\lambda)>0$ such that $t^-_{\lambda}(w_\lambda)w_\lambda\in \mathcal{M}^+_\lambda$. Next, we claim that $u_n\to w_\lambda$ in $W^{s,p}_{V}(\mathbb{R}^N)$ as $n\to\infty$. Indeed, if not, then $u_n\not\to w_\lambda$ in $W^{s,p}_{V}(\mathbb{R}^N)$ as $n\to\infty$, i.e., $u_n\rightharpoonup w_\lambda$ weakly in $W^{s,p}_{V}(\mathbb{R}^N)$ as $n\to\infty$ and also $t^-_{\lambda}(w_\lambda)u_n\rightharpoonup t^-_{\lambda}(w_\lambda) w_\lambda$ weakly in $W^{s,p}_{V}(\mathbb{R}^N)$ as $n\to\infty$. Therefore, we obtain
\begin{equation}\label{eq4.7}
\begin{cases}
 \displaystyle\liminf_{n\to\infty}\Psi_\lambda(t^-_{\lambda}(w_\lambda)u_n)>\Psi_\lambda(t^-_{\lambda}(w_\lambda)w_\lambda)\\ \hspace{4cm}\text{and}\\ \displaystyle\liminf_{n\to\infty}\mathcal{F}'_{\lambda,u_n}(t^+_\lambda(w_\lambda)) > \mathcal{F}'_{\lambda,w_\lambda}(t^+_\lambda(w_\lambda))=0.    
\end{cases}
\end{equation}
This shows that $\mathcal{F}'_{\lambda,u_n}(t^+_\lambda(w_\lambda))>0$ for $n$ large enough. As $u_n\in\mathcal{M}^-_\lambda$ for all $n\in N$, thus 1 is the point of maximum of $\mathcal{F}_{\lambda,u_n}$. It follows that
$\mathcal{F}'_{\lambda,u_n}(t)>0$ for all $t\in(0,1)$ and $\mathcal{F}'_{\lambda,u_n}(t)<0$ for all $t>1$. Hence we obtain $t^+_\lambda(w_\lambda)<1$. Now, using \eqref{eq4.7} and the fact that 1 is the global maximum point for $\mathcal{F}_{\lambda,u_n}$ lead us to conclude that
$$\Psi_\lambda(t^-_{\lambda}(w_\lambda)w_\lambda)< \liminf_{n\to\infty}\Psi_\lambda(t^-_{\lambda}(w_\lambda)u_n)\leq \liminf_{n\to\infty}\Psi_\lambda(u_n)=\Upsilon^-_\lambda,$$
which is a contradiction because $t^-_{\lambda}(w_\lambda)w_\lambda\in \mathcal{M}^+_\lambda$. This yields $u_n\to w_\lambda$ in $W^{s,p}_{V}(\mathbb{R}^N)$ as $n\to\infty$ and hence $\mathcal{F}'_{\lambda,w_\lambda}(1)=0$ and $\mathcal{F}''_{\lambda,w_\lambda}(1)\leq 0$. From Corollary \ref{cor3.6}, we have $\mathcal{M}^0_\lambda =\emptyset$ for $0<\lambda<\lambda_\ast$. Thus $w_\lambda\in\mathcal{M}^-_\lambda$ and $\Psi_\lambda(w_\lambda)=\mathcal{I}^-_\lambda(w_\lambda)=\Upsilon^-_\lambda$. This completes the proof.   
\end{proof}    
\begin{corollary}\label{cor4.3}
  Let $u\in\mathcal{M}^\pm_\lambda$, then there exists $\rho>0$ and a continuous function $t:B_\rho(0)\to (0,\infty)$ such that $t(0)=1$ and $t(y)(u+y)\in\mathcal{M}^\pm_\lambda $ for all $ y\in B_\rho(0)$, where $$B_\rho(0)=\big\{u\in W^{s,p}_{V}(\mathbb{R}^N):\|u\|=\rho\big\}. $$
\end{corollary}
\begin{proof}
  The proof is quite similar to [\cite{garain2023anisotropic}, Lemma 2.14] and hence omitted here.
\end{proof}
\begin{lemma}\label{lem4.4}
 Let $0<\lambda<\lambda_\ast$ and suppose $u_\lambda$ and $w_\lambda$ are minimizers of $\Psi_\lambda$ over $\mathcal{M}^+_\lambda$ and $\mathcal{M}^-_\lambda$ respectively. Then for every $v\in\mathcal{K}$, we have
 \begin{itemize}
     \item [(i)] there exists a $\epsilon_0>0$ such that $\Psi_\lambda(u_\lambda+\epsilon v)\geq\Psi_\lambda(u_\lambda)$ for every $0\leq\epsilon\leq\epsilon_0$.
     \item[(ii)] $t^-_\lambda(w_\lambda+\epsilon v)\to 1$ as $\epsilon\to 0^+$, where $t^-_\lambda(w_\lambda+\epsilon v)$ is a unique positive real number such that $t^-_\lambda(w_\lambda+\epsilon v)(w_\lambda+\epsilon v)\in \mathcal{M}^-_\lambda$.
 \end{itemize}
\end{lemma}
\begin{proof}
$(i)$ The proof is a consequence of Corollary \ref{cor4.3}. 

\noindent $(ii)$ Define a $\mathcal{C}^\infty$-function $\mathcal{H}:(0,\infty)\times\mathbb{R}^3 \to\mathbb{R}$ by $\mathcal{H}(t,\textit{a},\textit{b},\textit{c})=t^{p(m+1)-1}\textit{a}-\lambda t^{-\delta}\textit{b}-t^{\gamma-1}\textit{c}$, where $$\textit{a}=M\big(\|u\|^p\big)\|u\|^p,\textit{b}=\int_{\mathbb{R}^{N}}\alpha(x)|u|^{1-\delta}~\mathrm{d}x~\text{and}~\textit{c}=\int_{\mathbb{R}^{N}}\beta (x)|u|^\gamma~\mathrm{d}x. $$
    It follows from $w_\lambda\in\mathcal{M}^-_\lambda$ that $$\mathcal{H}(1,\textit{a},\textit{b},\textit{c})=\mathcal{F}'_{\lambda,w_\lambda}(1)=0~\text{and}~\frac{\mathrm{d}\mathcal{H}(1,\textit{a},\textit{b},\textit{c})}{\mathrm{d}t}=\mathcal{F}''_{\lambda,w_\lambda}(1)<0.$$ 
    By the hypothesis, we have $t^-_\lambda(w_\lambda+\epsilon v)(w_\lambda+\epsilon v)\in \mathcal{M}^-_\lambda$, i.e.,  $\mathcal{F}'_{\lambda,w_\lambda+\epsilon v}(t^-_\lambda(w_\lambda+\epsilon v))=0$. This shows that
    \begin{equation}\label{eq4.8}
      \mathcal{H}\bigg((t^-_\lambda(w_\lambda+\epsilon v),M\big(\|w_\lambda+\epsilon v\|^p\big)\|w_\lambda+\epsilon v\|^p,\int_{\mathbb{R}^{N}}\alpha(x)|w_\lambda+\epsilon v|^{1-\delta}~\mathrm{d}x,\int_{\mathbb{R}^{N}}\beta (x)|w_\lambda+\epsilon v|^\gamma~\mathrm{d}x\bigg)=0.     
    \end{equation}
    
    By implicit function theorem, there exist open neighbourhoods $A\subset(0,\infty)$ and $B\subset \mathbb{R}^3 $ containing 1 and $\bigg(M\big(\|w_\lambda\|^p\big)\|w_\lambda\|^p,\displaystyle\int_{\mathbb{R}^{N}}\alpha(x)|w_\lambda|^{1-\delta}~\mathrm{d}x,\displaystyle\int_{\mathbb{R}^{N}}\beta (x)|w_\lambda|^\gamma~\mathrm{d}x\bigg)$ raspectively such that for all $y\in B$, the equation $\mathcal{H}(t,y)=0$ has a unique solution  $t=\mathcal{G}(y)$ with $\mathcal{G}:B\to A$ being a smooth function. Choosing $\epsilon>0$ very small enough and
$$\bigg(M\big(\|w_\lambda+\epsilon v\|^p\big)\|w_\lambda+\epsilon v\|^p,\int_{\mathbb{R}^{N}}\alpha(x)|w_\lambda+\epsilon v|^{1-\delta}~\mathrm{d}x,\int_{\mathbb{R}^{N}}\beta (x)|w_\lambda+\epsilon v|^\gamma~\mathrm{d}x\bigg)\in B ,$$ then we conclude that \eqref{eq4.8} has a unique solution
$$\mathcal{G}\bigg(M\big(\|w_\lambda+\epsilon v\|^p\big)\|w_\lambda+\epsilon v\|^p,\int_{\mathbb{R}^{N}}\alpha(x)|w_\lambda+\epsilon v|^{1-\delta}~\mathrm{d}x,\int_{\mathbb{R}^{N}}\beta (x)|w_\lambda+\epsilon v|^\gamma~\mathrm{d}x\bigg)=t^-_\lambda(w_\lambda+\epsilon v) .$$
The continuity of $\mathcal{G}$ infers that $t^-_\lambda(w_\lambda+\epsilon v)\to 1$ as $\epsilon\to 0^+$.This completes the proof. 
\end{proof}
\begin{lemma}\label{lem4.5}
    Let $0<\lambda<\lambda_\ast$, then the following results hold:
    \begin{itemize}
        \item [(i)] if $u_\lambda$ is a minimizer of $\Psi_\lambda$ over $\mathcal{M}^+_\lambda$, then $u_\lambda>0$ a.e. in $\mathbb{R}^N$, $\alpha(x)u_\lambda^{-\delta}v\in L^1(\mathbb{R}^N)$ and it satisfies
\begin{equation}\label{eq4.997}
         \langle{L(u_\lambda),v}\rangle-\lambda \int_{\mathbb{R}^{N}} \alpha(x)u_\lambda^{-\delta}v~\mathrm{d}x-\int_{\mathbb{R}^{N}}\beta(x)u_\lambda^{\gamma-1}v~\mathrm{d}x\geq 0~\text{for all}~v\in\mathcal{K}.   
\end{equation}
\item[(ii)] if $w_\lambda$ is a minimizer of $\Psi_\lambda$ over $\mathcal{M}^-_\lambda$, then $w_\lambda>0$ a.e. in $\mathbb{R}^N$, $\alpha(x)w_\lambda^{-\delta}v\in L^1(\mathbb{R}^N)$ and it satisfies
\begin{equation}\label{eq4.1000}
\langle{L(w_\lambda),v}\rangle-\lambda \int_{\mathbb{R}^{N}} \alpha(x)w_\lambda^{-\delta}v~\mathrm{d}x-\int_{\mathbb{R}^{N}}\beta(x)w_\lambda^{\gamma-1}v~\mathrm{d}x\geq 0~\text{for all}~v\in\mathcal{K}.    
\end{equation}
    \end{itemize}    
\end{lemma}
\begin{proof}
        $(i)$ Since $u_\lambda$ is a minimizer of $\Psi_\lambda$ on $\mathcal{M}^+_\lambda$, therefore $u_\lambda\geq 0$ a.e. in $\mathbb{R}^{N}$. We claim that $u_\lambda>0$ for a.a. $x\in\mathbb{R}^{N}$. Indeed, if not, then let there exists a set $\mathbb{C}$ of positive measure such that $u_\lambda=0$ in $\mathbb{C}$. Choose $v\in W^{s,p}_{V}(\mathbb{R}^N)$, $v>0$ and $\epsilon\in(0,\epsilon_0)$ as in Lemma \ref{lem4.4} such that $(u_\lambda+\epsilon v)^{1-\delta}>u^{1-\delta}_\lambda$ a.e. in $\mathbb{R}^N\setminus\mathbb{C}$. From Lemma \ref{lem4.4}-$(i)$, we have 
        
        $$\hspace{-11cm}0\leq \frac{\Psi_\lambda(u_\lambda+\epsilon v)-\Psi_\lambda(u_\lambda)}{\epsilon} $$ 
        $$< \frac{1}{p\epsilon}\bigg[\widehat{M}\big(\|u_\lambda+\epsilon v\|^p\big)-\widehat{M}\big(\|u_\lambda\|^p\big)\bigg]-\frac{\lambda}{(1-\delta)\epsilon^\delta}\int_\mathbb{C}\alpha(x)v^{1-\delta}~\mathrm{d}x-\frac{1}{\gamma\epsilon}\int_{\mathbb{R}^N}\beta(x) \big((u_\lambda+\epsilon v)^{\gamma}-u^{\gamma}_\lambda\big)~\mathrm{d}x.$$
        This yields $$0\leq \frac{\Psi_\lambda(u_\lambda+\epsilon v)-\Psi_\lambda(u_\lambda)}{\epsilon} \to-\infty~\text{as}~\epsilon\to 0^+, $$
        which is a contradiction. Thus we have  $u_\lambda>0$ a.e. in $\mathbb{R}^N$. Further, take $v\in\mathcal{K}$ and choose a decreasing sequence $\{\epsilon_n\}_{n\in \mathbb{N}}\subset (0,1]$ such that $\epsilon_n\to 0$ as $n\to\infty$. For each $n\in \mathbb{N}$, define a sequence of non-negative measurable functions $\{g_n(x)\}_{n\in \mathbb{N}}$ by
        $$g_n(x)=\alpha(x)\bigg(\frac{(u_\lambda+\epsilon_n v)^{1-\delta}-u^{1-\delta}_\lambda}{\epsilon_n}\bigg). $$ By Mean Value theorem, we have $$\lim_{n\to\infty}g_n(x)=(1-\delta)\alpha(x)u^{-\delta}_\lambda v~\text{for a.e.}~x\in \mathbb{R}^N.$$
        Consequently, by Fatou's Lemma, we obtain 
        \begin{equation}\label{eq4.9}
          \int_{\mathbb{R}^N} \alpha(x)u^{-\delta}_\lambda v~\mathrm{d}x\leq \frac{1}{1-\delta}\liminf_{n\to\infty} \int_{\mathbb{R}^N} \alpha(x)\bigg(\frac{(u_\lambda+\epsilon_n v)^{1-\delta}-u^{1-\delta}_\lambda}{\epsilon_n}\bigg)~\mathrm{d}x.
        \end{equation}
        Using Lemma \ref{lem4.4}-$(i)$ for n large enough, we have
        \begin{equation}\label{eq4.666}
          0\leq \frac{\Psi_\lambda(u_\lambda+\epsilon_n v)-\Psi_\lambda(u_\lambda)}{\epsilon_n}.   
        \end{equation} 
        On simplifying \eqref{eq4.666}, we have
        \begin{equation}\label{eq4.10}
        \begin{split}
          \frac{\lambda}{1-\delta}\int_{\mathbb{R}^N} \alpha(x)\bigg(\frac{(u_\lambda+\epsilon_n v)^{1-\delta}-u^{1-\delta}_\lambda}{\epsilon_n}\bigg)~\mathrm{d}x\leq \frac{1}{p}\Bigg[\frac{\widehat{M}\big(\|u_\lambda+\epsilon_n v\|^p\big)-\widehat{M}\big(\|u_\lambda\|^p}{\epsilon_n}\Bigg]\\&\hspace{-8cm}-\frac{1}{\gamma}\int_{\mathbb{R}^N} \beta(x)\bigg(\frac{(u_\lambda+\epsilon_n v)^{\gamma}-u^{\gamma}_\lambda}{\epsilon_n}\bigg)~\mathrm{d}x.  
        \end{split}  
        \end{equation}
        By the mean value theorem and Lebesgue-dominated convergence theorem, we have
        \begin{equation}\label{eq4.1001}
         \frac{1}{p}\Bigg[\frac{\widehat{M}\big(\|u_\lambda+\epsilon_n v\|^p\big)-\widehat{M}\big(\|u_\lambda\|^p}{\epsilon_n}\Bigg]\to \langle{L(u_\lambda),v}\rangle ~\text{as} ~n\to\infty  
        \end{equation}
        and
        \begin{equation}\label{eq4.1002}
          \frac{1}{\gamma}\int_{\mathbb{R}^N} \beta(x)\bigg(\frac{(u_\lambda+\epsilon v)^{\gamma}-u^{\gamma}_\lambda}{\epsilon_n}\bigg)~\mathrm{d}x\to \int_{\mathbb{R}^{N}}\beta(x)u_\lambda^{\gamma-1}v~\mathrm{d}x~\text{as} ~n\to\infty.  
        \end{equation}
        Taking limit $n\to\infty$ on both the sides of \eqref{eq4.10} and using \eqref{eq4.1001} and \eqref{eq4.1002}, we deduce from \eqref{eq4.9} that $ \alpha(x)u_\lambda^{-\delta}v\in L^1(\mathbb{R}^N)$
        and \eqref{eq4.997} holds.
        
        \noindent $(ii)$ According to Lemma \ref{lem4.4} $(ii)$, we have $t^-_\lambda(w_\lambda+\epsilon v)\to 1$ as $\epsilon\to 0^+$, where $t^-_\lambda(w_\lambda+\epsilon v)(w_\lambda+\epsilon v)\in \mathcal{M}^-_\lambda$ for $v\in\mathcal{K}$ and $0\leq\epsilon\leq\epsilon_0$. Since $w_\lambda$ is a minimizer of $\Psi_\lambda$ on $\mathcal{M}^-_\lambda$, therefore $w_\lambda\geq 0$ a.e. in $\mathbb{R}^{N}$ and 
        \begin{equation}\label{eq4.11}
          \Psi_\lambda(w_\lambda)\leq \Psi_\lambda(t^-_\lambda(w_\lambda+\epsilon v)(w_\lambda+\epsilon v))~\text{for all}~0\leq\epsilon\leq\epsilon_0. 
        \end{equation}
        Let us prove that $w_\lambda>0$ for a.a. $x\in\mathbb{R}^N$. Indeed, if not, then let $w_\lambda=0$ in $\mathbb{C}$, where the measure of the set $\mathbb{C}$ is positive. Furthermore, choose $v\in W^{s,p}_{V}(\mathbb{R}^N)$ with $v>0$ and $\epsilon\in(0,\epsilon_0)$ small enough such that $\big(t^-_\lambda(w_\lambda+\epsilon v)(w_\lambda+\epsilon v)\big)^{1-\delta}>\big(t^-_\lambda(w_\lambda+\epsilon v)(w_\lambda)\big)^{1-\delta}$ a.e. in $\mathbb{R}^N\setminus\mathbb{C}$. From \eqref{eq4.11} and using the fact that 1 is the global maximum point for $\mathcal{F}_{\lambda,w_\lambda}$, we have the following.
        \begin{equation}\label{eq4.12}
          \Psi_\lambda(w_\lambda)\geq \Psi_\lambda(t^-_\lambda(w_\lambda+\epsilon v)w_\lambda).  
        \end{equation}
        Combining \eqref{eq4.11} and \eqref{eq4.12}, we have
        \begin{equation}\label{eq4.999}
         \Psi_\lambda(t^-_\lambda(w_\lambda+\epsilon v)(w_\lambda+\epsilon v))\geq  \Psi_\lambda(t^-_\lambda(w_\lambda+\epsilon v)w_\lambda) ~\text{for} ~\epsilon\in(0,\epsilon_0).
        \end{equation}       
        Now using \eqref{eq4.999} and applying a similar strategy as in case-$(i)$, we can deduce that $w_\lambda>0$ a.e. in $\mathbb{R}^N$, $\alpha(x)w_\lambda^{-\delta}v\in L^1(\mathbb{R}^N)$ and \eqref{eq4.1000} holds.   
\end{proof}
\begin{theorem}\label{thm4.6}
 Let $0<\lambda<\lambda_*$, then the minimizers $u_\lambda$ and $w_\lambda$ for the functional $\Psi_\lambda$ on $\mathcal{M}^+_\lambda$ and $\mathcal{M}^-_\lambda$, respectively are weak solutions of \eqref{main problem}.
\end{theorem}
\begin{proof}
 To prove $u_\lambda$ is a weak solution of \eqref{main problem}, choose $v\in W^{s,p}_{V}(\mathbb{R}^N)$ and define $\phi_\epsilon=u_\lambda+\epsilon v$, then for each $\epsilon>0$ be given $\phi^+_\epsilon\in\mathcal{K}$. Now by Lemma \ref{lem4.5}-$(i)$, we have 
$$\langle{L(u_\lambda),\phi^+_\epsilon}\rangle-\lambda \int_{\mathbb{R}^{N}} \alpha(x)u_\lambda^{-\delta}\phi^+_\epsilon~\mathrm{d}x-\int_{\mathbb{R}^{N}}\beta(x)u_\lambda^{\gamma-1}\phi^+_\epsilon~\mathrm{d}x\geq 0 .$$
Replacing $\phi^+_\epsilon=\phi_\epsilon+\phi^-_\epsilon$ in the above inequality, we obtain
$$\langle{L(u_\lambda),\phi_\epsilon+\phi^-_\epsilon}\rangle-\lambda \int_{\mathbb{R}^{N}} \alpha(x)u_\lambda^{-\delta}(\phi_\epsilon+\phi^-_\epsilon)~\mathrm{d}x-\int_{\mathbb{R}^{N}}\beta(x)u_\lambda^{\gamma-1}(\phi_\epsilon+\phi^-_\epsilon)~\mathrm{d}x\geq 0 .$$
Define $\mathcal{D}_{\epsilon}=\{x\in\mathbb{R}^{N}:~\phi_\epsilon(x)\leq 0\}$ and $\mathcal{D}^c_{\epsilon}=\{x\in\mathbb{R}^{N}:~\phi_\epsilon(x)>0\}$. After some straightforward calculations, we have 
\begin{align*} \quad 0&\leq \Bigg[M\big(\|u_\lambda\|^p\big)\|u_\lambda\|^p-\lambda \int_{\mathbb{R}^{N}} \alpha(x)u_\lambda^{1-\delta}~\mathrm{d}x-\int_{\mathbb{R}^{N}}\beta(x)u_\lambda^{\gamma}~\mathrm{d}x \Bigg]\\
&\qquad+\epsilon\Bigg[\langle{L(u_\lambda),v}\rangle-\lambda \int_{\mathbb{R}^{N}} \alpha(x)u_\lambda^{-\delta}v~\mathrm{d}x-\int_{\mathbb{R}^{N}}\beta(x)u_\lambda^{\gamma-1}v~\mathrm{d}x \Bigg]\\
&\qquad+M\big(\|u_\lambda\|^p\big)\bigg[\int_{\mathbb{R}^{N}}\int_{\mathbb{R}^{N}}\frac{|u_\lambda(x)-u_\lambda(y)|^{p-2}(u_\lambda(x)-u_\lambda(y))(\phi^-_\epsilon(x)-\phi^-_\epsilon(y))}{|x-y|^{N+sp}}~\mathrm{d}x\mathrm{d}y-\int_{\mathcal{D}_{\epsilon}}V(x)u_\lambda^{p-1}(u_\lambda+\epsilon v)~\mathrm{d}x\bigg]\\
&\qquad+\lambda \int_{\mathcal{D}_{\epsilon}} \alpha(x)u_\lambda^{-\delta}(u_\lambda+\epsilon v)~\mathrm{d}x+\int_{\mathcal{D}_{\epsilon}}\beta(x)u_\lambda^{\gamma-1}(u_\lambda+\epsilon v)~\mathrm{d}x.\end{align*}
This yields
\begin{equation}\label{eq4.13}
    \begin{split}
      0\leq\epsilon\Bigg[\langle{L(u_\lambda),v}\rangle-\lambda \int_{\mathbb{R}^{N}} \alpha(x)u_\lambda^{-\delta}v~\mathrm{d}x-\int_{\mathbb{R}^{N}}\beta(x)u_\lambda^{\gamma-1}v~\mathrm{d}x \Bigg]\\
      &\hspace{-8cm}+M\big(\|u_\lambda\|^p\big)\bigg[\int_{\mathbb{R}^{N}}\int_{\mathbb{R}^{N}}\frac{|u_\lambda(x)-u_\lambda(y)|^{p-2}(u_\lambda(x)-u_\lambda(y))(\phi^-_\epsilon(x)-\phi^-_\epsilon(y))}{|x-y|^{N+sp}}~\mathrm{d}x\mathrm{d}y-\epsilon\int_{\mathcal{D}_{\epsilon}}V(x)u_\lambda^{p-1}v~\mathrm{d}x\bigg]\\
      &\hspace{-7.9cm}+\int_{\mathcal{D}_{\epsilon}}\beta(x)u_\lambda^{\gamma-1}(u_\lambda+\epsilon v)~\mathrm{d}x , 
    \end{split}
\end{equation}
 where the above inequality is obtained by using the fact that $u_\lambda\in\mathcal{M}^+_\lambda$, i.e., 
$$ M\big(\|u_\lambda\|^p\big)\|u_\lambda\|^p-\lambda \int_{\mathbb{R}^{N}} \alpha(x)u_\lambda^{1-\delta}~\mathrm{d}x-\int_{\mathbb{R}^{N}}\beta(x)u_\lambda^{\gamma}~\mathrm{d}x =0. $$  
Further, we also have used 
  $$\int_{\mathcal{D}_{\epsilon}} \alpha(x)u_\lambda^{-\delta}(u_\lambda+\epsilon v)~\mathrm{d}x\leq 0~\text{and}~\int_{\mathcal{D}_{\epsilon}}V(x)u_\lambda^{p}~\mathrm{d}x\geq 0.$$  
Define $$\mathcal{I_{\lambda,\epsilon}}=\int_{\mathbb{R}^{N}}\int_{\mathbb{R}^{N}}\frac{|u_\lambda(x)-u_\lambda(y)|^{p-2}(u_\lambda(x)-u_\lambda(y))(\phi^-_\epsilon(x)-\phi^-_\epsilon(y))}{|x-y|^{N+sp}}~\mathrm{d}x\mathrm{d}y,$$
then by the symmetry of the fractional kernel and after some simple computations, we obtain
\begin{align*}\mathcal{I_{\lambda,\epsilon}}&=\int_{\mathcal{D}_{\epsilon}}\int_{\mathcal{D}_{\epsilon}}\frac{|u_\lambda(x)-u_\lambda(y)|^{p-2}(u_\lambda(x)-u_\lambda(y))(\phi^-_\epsilon(x)-\phi^-_\epsilon(y))}{|x-y|^{N+sp}}~\mathrm{d}x\mathrm{d}y\\
&\qquad+2 \int_{\mathcal{D}_{\epsilon}}\int_{\mathcal{D}^c_{\epsilon}}\frac{|u_\lambda(x)-u_\lambda(y)|^{p-2}(u_\lambda(x)-u_\lambda(y))(\phi^-_\epsilon(x)-\phi^-_\epsilon(y))}{|x-y|^{N+sp}}~\mathrm{d}x\mathrm{d}y \\
&\leq -\epsilon\bigg[\int_{\mathcal{D}_{\epsilon}}\int_{\mathcal{D}_{\epsilon}}\frac{|u_\lambda(x)-u_\lambda(y)|^{p-2}(u_\lambda(x)-u_\lambda(y))(v(x)-v(y))}{|x-y|^{N+sp}}~\mathrm{d}x\mathrm{d}y\\
&\qquad+2 \int_{\mathcal{D}_{\epsilon}}\int_{\mathcal{D}^c_{\epsilon}}\frac{|u_\lambda(x)-u_\lambda(y)|^{p-2}(u_\lambda(x)-u_\lambda(y))(v(x)-v(y))}{|x-y|^{N+sp}}~\mathrm{d}x\mathrm{d}y \bigg] \\
&\leq 2\epsilon \int_{\mathcal{D}_{\epsilon}}\int_{\mathbb{R}^{N}}\frac{|u_\lambda(x)-u_\lambda(y)|^{p-1}|v(x)-v(y)|}{|x-y|^{N+sp}}~\mathrm{d}x\mathrm{d}y.\end{align*}
Now using the fact that $u_\lambda$ is bounded in $ W^{s,p}_{V}(\mathbb{R}^N)$ and Holder's inequality, we infer that there exists a positive constant $C>0$ such that 
\begin{equation}\label{eq4.16}
 \int_{\mathcal{D}_{\epsilon}}\int_{\mathbb{R}^{N}}\frac{|u_\lambda(x)-u_\lambda(y)|^{p-1}|v(x)-v(y)|}{|x-y|^{N+sp}}~\mathrm{d}x\mathrm{d}y\leq C\Bigg[\int_{\mathcal{D}_{\epsilon}}\int_{\mathbb{R}^{N}}\frac{|v(x)-v(y)|^p}{|x-y|^{N+sp}}~\mathrm{d}x\mathrm{d}y \Bigg] ^{\frac{1}{p}}.
\end{equation}
From the above inequality and \eqref{eq4.16}, we obtain
\begin{equation}\label{eq4.17}
\mathcal{I_{\lambda,\epsilon}}\leq 2C\epsilon \Bigg[\int_{\mathcal{D}_{\epsilon}}\int_{\mathbb{R}^{N}}\frac{|v(x)-v(y)|^p}{|x-y|^{N+sp}}~\mathrm{d}x\mathrm{d}y \Bigg] ^{\frac{1}{p}}.  
\end{equation}
Since $\frac{|v(x)-v(y)|^p}{|x-y|^{N+sp}}\in L^1(\mathbb{R}^{2N})$, thus for any $\zeta>0$ there exists $R_\zeta$ large enough such that
$$\int_{\textit{supp}(v)}\int_{\mathbb{R}^{N}\setminus B_{R_\zeta}(0)}\frac{|v(x)-v(y)|^p}{|x-y|^{N+sp}}~\mathrm{d}x\mathrm{d}y<\frac{\zeta}{2} .$$
It follows from $\mathcal{D}_{\epsilon}\subset \textit{supp}(v) $ and $\big|\mathcal{D}_{\epsilon}\times B_{R_\zeta}(0)\big|\to 0$ as $\epsilon\to 0^+$ that there exists a $\rho_{\zeta}>0$ and $\epsilon_\zeta>0$ such that
$$ \big|\mathcal{D}_{\epsilon}\times B_{R_\zeta}(0)\big|<\rho_{\zeta}~\text{and}~\int_{\mathcal{D}_{\epsilon}}\int_{B_{R_\zeta}(0)}\frac{|v(x)-v(y)|^p}{|x-y|^{N+sp}}~\mathrm{d}x\mathrm{d}y<\frac{\zeta}{2} ~\text{for}~\epsilon\in(0,\epsilon_\zeta).$$
This shows that for any $\epsilon\in(0,\epsilon_\zeta)$, we have 
$$ \int_{\mathcal{D}_{\epsilon}}\int_{\mathbb{R}^{N}}\frac{|v(x)-v(y)|^p}{|x-y|^{N+sp}}~\mathrm{d}x\mathrm{d}y <\zeta .$$
Hence 
\begin{equation}\label{eq4.18}
    \lim_{\epsilon\to 0^+} \int_{\mathcal{D}_{\epsilon}}\int_{\mathbb{R}^{N}}\frac{|v(x)-v(y)|^p}{|x-y|^{N+sp}}~\mathrm{d}x\mathrm{d}y=0.
\end{equation}
 Observe that $u_\lambda\leq -\epsilon v$ on $\mathcal{D}_{\epsilon}$ and hence by Lemma \ref{lemma2.3} there exists a positive constant $c>0$ such that
 $$\bigg|\int_{\mathcal{D}_{\epsilon}}\beta(x)u_\lambda^{\gamma-1}(u_\lambda+\epsilon v)~\mathrm{d}x\bigg|\leq \int_{\mathcal{D}_{\epsilon}}|\beta(x)u_\lambda^{\gamma}+\beta(x)\epsilon u_\lambda^{\gamma-1}v)|~\mathrm{d}x\leq 2c\epsilon^\gamma\|\beta\|_{L^\infty(\mathbb{R}^{N})}\|v\|^\gamma . $$
 This yields
 \begin{equation}\label{eq4.19}
    \lim_{\epsilon\to 0^+}\frac{1}{\epsilon} \int_{\mathcal{D}_{\epsilon}}\beta(x)u_\lambda^{\gamma-1}(u_\lambda+\epsilon v)~\mathrm{d}x=0.
 \end{equation}
  Dividing $\epsilon$ on both the sides of \eqref{eq4.13} and using that $|\mathcal{D}_{\epsilon}|\to 0$ as $\epsilon\to 0^+$, \eqref{eq4.17},~\eqref{eq4.18} and \eqref{eq4.19} hold true, we notice that 
$$\langle{L(u_\lambda),v}\rangle-\lambda \int_{\mathbb{R}^{N}} \alpha(x)u_\lambda^{-\delta}v~\mathrm{d}x-\int_{\mathbb{R}^{N}}\beta(x)u_\lambda^{\gamma-1}v~\mathrm{d}x\geq 0. $$
Due to the arbitrariness of $v$, we deduce that $u_\lambda$ is a weak solution to \eqref{main problem}. Similarly, we can prove that $w_\lambda$ is a weak solution to \eqref{main problem}. This completes the proof.
\end{proof}
\section{Existence of solutions for \texorpdfstring{$\lambda=\lambda_\ast$}{Lg}}\label{sec5}
In this section, we study the existence of solutions to \eqref{main problem} for $\lambda=\lambda_\ast$ and characterize the Nehari submanifold $\mathcal{M}^0_{\lambda_\ast}$, which is non-empty.
\begin{lemma}\label{lem5.1}
  Let $u\in \mathcal{M}^0_{\lambda_\ast}$, then for every $v\in W^{s,p}_{V}(\mathbb{R}^N)$ the following result holds 
  \begin{equation}\label{eq5.777}
 p(m+1)\langle{L(u),v}\rangle
-\lambda_\ast \int_{\mathbb{R}^{N}} \alpha(x)u^{-\delta}v~\mathrm{d}x-\int_{\mathbb{R}^{N}}\beta(x)u^{\gamma-1}v~\mathrm{d}x= 0.
\end{equation}
In particular, $(\mathcal{E}_{\lambda_\ast})$ has no solution in $\mathcal{M}^0_{\lambda_\ast}$ .
\end{lemma}
\begin{proof}
    Let us rewrite $\lambda(u)$ by $\lambda(u)=C(a,\gamma,\delta,m,p)g(u)h(u)$, where we define $g$ and $h$ by
    \begin{equation}\label{eq5.99}
      g(u)=\frac{1}{\displaystyle\int_{\mathbb{R}^{N}}\alpha (x)|u|^{1-\delta}~\mathrm{d}x} ~\text{and}~h(u)=\frac{\Big[\|u\|^{p(m+1)}\Big]^{\frac{\gamma+\delta-1}{\gamma-p(m+1)}}}{\bigg[\displaystyle\int_{\mathbb{R}^{N}}\beta (x)|u|^\gamma~\mathrm{d}x\bigg]^{\frac{p(m+1)+\delta-1}{\gamma-p(m+1)}}}.  
    \end{equation}
    Let $u\in\mathcal{M}^0_{\lambda_\ast} $ and $v\in\mathcal{K}$. By Theorem \ref{thm3.1} and Corollary \ref{cor3.3}, the map $u\mapsto\displaystyle\int_{\mathbb{R}^{N}}\beta (x)|u|^\gamma~\mathrm{d}x$ is continuous and $\displaystyle\int_{\mathbb{R}^{N}}\beta (x)|u|^\gamma~\mathrm{d}x>0$. It follows that $\displaystyle\int_{\mathbb{R}^{N}}\beta (x)|u+tv|^\gamma~\mathrm{d}x>0$ for $t>0$ small enough and hence $h(u+tv)$ is well-defined for $t>0$ small enough. Thus $\langle{h'(u),v}\rangle$ exists finitely. Now by Corollary \ref{cor3.6}, we have $\lambda(u)=\lambda_\ast$. This shows that $\lambda(u+tv) -\lambda(u)=\lambda(u+tv)-\lambda_\ast\geq 0~\text{for all}~t>0 ~\text{small enough}.$
    Hence 
    \begin{equation}\label{eq5.1}
      (h(u+tv)-h(u))g(u+tv)\geq -h(u)(g(u+tv)-g(u)).  
    \end{equation}
    Dividing $t$ on both the sides of \eqref{eq5.1} and applying limit inferior as $t\to 0^+$, we get
    \begin{equation}\label{eq5.2}
     h(u)\bigg[\int_{\mathbb{R}^{N}}\alpha (x)|u|^{1-\delta}~\mathrm{d}x\bigg]^{-2}\liminf_{t\to 0^+}\int_{\mathbb{R}^{N}}\alpha(x)\bigg(\frac{|u+tv|^{1-\delta}-|u|^{1-\delta}}{t}\bigg)~\mathrm{d}x \leq \langle{h'(u),v}\rangle  g(u)<\infty.   
    \end{equation}
    By Mean Value theorem, there exists $\theta\in(0,1)$ such that 
    $$\alpha(x)\bigg(\frac{|u+tv|^{1-\delta}-|u|^{1-\delta}}{t}\bigg)=(1-\delta)\alpha(x)(u+\theta tv)^{-\delta}v \geq 0$$
    and 
    \begin{equation}\label{eq5.999}
     \lim_{t\to 0^+}\alpha(x)\bigg(\frac{|u+tv|^{1-\delta}-|u|^{1-\delta}}{t}\bigg)=H(x)~\text{(say)}=\begin{cases}
        0~\text{if}~v=0,u>0;\\
        (1-\delta)\alpha(x)u^{-\delta}v~\text{if}~v>0,u>0;\\
        \infty~\text{if}~v>0,u=0.
        \end{cases}
    \end{equation}
    Now, by applying Fatou's lemma, we have
    \begin{equation}\label{eq5.3}
    \int_{\mathbb{R}^{N}}H(x)~\mathrm{d}x\leq \liminf_{t\to 0^+}\int_{\mathbb{R}^{N}}\alpha(x)\bigg(\frac{|u+tv|^{1-\delta}-|u|^{1-\delta}}{t}\bigg)~\mathrm{d}x.   
    \end{equation}
    It is clear from \eqref{eq5.2} and \eqref{eq5.3} that $0\leq \displaystyle\int_{\mathbb{R}^{N}}H(x)~\mathrm{d}x<\infty $. Hence we must have $H(x)=(1-\delta)\alpha(x)u^{-\delta}v$ and $u>0$ a.e. in $\mathbb{R}^{N}$. Choosing $v>0$, we obtain $0<\displaystyle\int_{\mathbb{R}^{N}}\alpha(x)u^{-\delta}v~\mathrm{d}x<\infty$. It follows that $\langle{g'(u),v}\rangle$ exists and is given by
    \begin{equation}\label{eq5.5}
      \langle{g'(u),v}\rangle=-(1-\delta)\bigg[\int_{\mathbb{R}^{N}}\alpha (x)|u|^{1-\delta}~\mathrm{d}x\bigg]^{-2}\int_{\mathbb{R}^{N}}\alpha(x)u^{-\delta}v~\mathrm{d}x.  
    \end{equation}
    Taking into account \eqref{eq5.99}, \eqref{eq5.2}, \eqref{eq5.3}, \eqref{eq5.5} and the fact that $u\in\mathcal{M}^0_{\lambda_\ast}$, we can easily obtain 
\begin{equation}\label{eq5.6}
 p(m+1)\langle{L(u),v}\rangle
-\lambda_\ast \int_{\mathbb{R}^{N}} \alpha(x)u^{-\delta}v~\mathrm{d}x-\int_{\mathbb{R}^{N}}\beta(x)u^{\gamma-1}v~\mathrm{d}x\geq 0,~\forall~v\in \mathcal{K}. 
\end{equation}
Choose $v\in W^{s,p}_{V}(\mathbb{R}^N)$ and define $\phi_\epsilon=u+\epsilon v$, then for each $\epsilon>0$ be given $\phi^+_\epsilon\in\mathcal{K}$. Now, the result follows by replacing $\phi_\epsilon$ in place of $v$ in \eqref{eq5.6} and applying a similar strategy as in Theorem \ref{thm4.6}.\\
Next, we have to show that $(\mathcal{E}_{\lambda_\ast})$ has no solution in $\mathcal{M}^0_{\lambda_\ast}$. On the contrary, let it have a solution $u\in\mathcal{M}^0_{\lambda_\ast}$. Then from the weak formulation definition and \eqref{eq5.777}, we deduce that 
$$ \int_{\mathbb{R}^{N}}\big(\lambda_\ast(p(m+1)+\delta-1)\alpha(x)u^{-\delta}+(p(m+1)-\gamma)\beta(x)u^{\gamma-1}\big)v~\mathrm{d}x=0~\text{for all}~v\in W^{s,p}_{V}(\mathbb{R}^N). $$
This yields
$$\lambda_\ast(p(m+1)+\delta-1)\alpha(x)u^{-\delta}= (\gamma-p(m+1))\beta(x)u^{\gamma-1}~\text{a.e. in} ~\mathbb{R}^{N}. $$
Due to the sign-changing behaviour of $\beta$, we have two situations, i.e., either $\beta(x)\leq 0$ or $\beta(x)> 0$. If $\beta(x)\leq 0$ in some positive measure set $\mathcal{D}\subset \mathbb{R}^N$, then it is easy to see that $\lambda_\ast(p(m+1)+\delta-1)\alpha(x)u^{-\delta}\leq 0$ in $\mathcal{D}$, which is a contradiction. Further, if $\beta(x)>0$ a.e. in $\mathbb{R}^N$, then from $(H4)$ we obtain
$$u=\Bigg[\frac{\lambda_\ast(p(m+1)+\delta-1)\alpha(x)}{(\gamma-p(m+1))\beta(x)}\Bigg]^{\frac{1}{\gamma+\delta-1}}\notin W^{s,p}_{V}(\mathbb{R}^N), $$ 
which is again a contradiction. This completes the proof.  
\end{proof}
\begin{corollary}\label{cor5.2}
    The Nehari submanifold set $\mathcal{M}^0_{\lambda_\ast}$ is compact.
\end{corollary}
\begin{proof}
    Choose a sequence $\{u_n\}_{n\in \mathbb{N}}$ in $\mathcal{M}^0_{\lambda_\ast}$, then by Corollary \ref{cor3.3}, $\{u_n\}_{n\in \mathbb{N}}$ is bounded. Hence up to a subsequence $u_n\rightharpoonup u$ weakly in $W^{s,p}_{V}(\mathbb{R}^N)$ as $n\to\infty$. Similar to Lemma \ref{lem4.2}, we can obtain $u>0$ and $\displaystyle\int_{\mathbb{R}^{N}}\beta(x)|u|^{\gamma}~\mathrm{d}x>0$. Now we claim that $u_n\to u$ in $W^{s,p}_{V}(\mathbb{R}^N)$ as $n\to\infty$. Indeed, if not, then we have
    $\lambda_\ast=\displaystyle\liminf_{n\to\infty}\lambda(u_n)>\lambda(u), $
    which is a contradiction. Therefore, $u_n\to u$ in $W^{s,p}_{V}(\mathbb{R}^N)$ as $n\to\infty$ and $\mathcal{F}'_{\lambda_*,u}(1)=\mathcal{F}''_{\lambda_*,u}(1)=0$. Thus, we get $u\in \mathcal{M}^0_{\lambda_\ast}$ and hence the result follows.
\end{proof}
Define $t_{\lambda_\ast}:\overline{\hat{\mathcal{M}}_{\lambda_\ast}}\setminus\{0\}\to \mathbb{R}$ and $ s_{\lambda_\ast}:\overline{\hat{\mathcal{M}}_{\lambda_\ast}\cup\hat{\mathcal{M}}^+_{\lambda_\ast}}\to\mathbb{R}$ by 

  $$t_{\lambda_\ast}(u)=\begin{cases}
     t^-_{\lambda_\ast}(u)~\text{if}~u\in\hat{\mathcal{M}}_{\lambda_\ast},\\
     t^0_{\lambda_\ast}(u)~\text{otherwise}.
  \end{cases} 
  ~\text{and}~s_{\lambda_\ast}(w)=\begin{cases}
     t^+_{\lambda_\ast}(w)~\text{if}~w\in\hat{\mathcal{M}}_{\lambda_\ast}\cup\hat{\mathcal{M}}^+_{\lambda_\ast},\\
     t^0_{\lambda_\ast}(w)~\text{otherwise}. 
  \end{cases} .$$
\begin{corollary}\label{cor5.3}
    Let $u\notin \hat{\mathcal{M}}^+_{\lambda_\ast}$, then the following results hold:
    \begin{itemize}
        \item [(i)] $$\lim_{\lambda\uparrow{\lambda_\ast}}t^-_\lambda(u)=t_{\lambda_\ast}(u)~ \text{and} ~ \lim_{\lambda\uparrow {\lambda_\ast}}t^+_\lambda(u)=s_{\lambda_\ast}(u).$$
        \item[(ii)] $$\lim_{\lambda\uparrow {\lambda_\ast}}\mathcal{I}^-_{\lambda}(u)=\Psi_{\lambda_\ast}(t_{\lambda_\ast}(u)u)~ \text{and} ~ \lim_{\lambda\uparrow {\lambda_\ast}}\mathcal{I}^+_{\lambda}(u)=\Psi_{\lambda_\ast}(s_{\lambda_\ast}(u)u).$$
    \end{itemize}
\end{corollary}
\begin{proof}
    It follows immediately from Lemma \ref{lem3.7}.
\end{proof}
The results below, which may be directly obtained from \cite{silva2018local}, are important to show the existence of solutions when $\lambda\geq \lambda_\ast $.
\begin{proposition}\label{prop5.4}
    The following results hold:
    \begin{itemize}
        \item[(i)] $\overline{\hat{\mathcal{M}}_{\lambda_\ast}\cup\hat{\mathcal{M}}^+_{\lambda_\ast}}=\overline{\hat{\mathcal{M}}_{\lambda_\ast}}\cup \overline{\hat{\mathcal{M}}^+_{\lambda_\ast}}\cup \{tu:t>0,~u\in\mathcal{M}^0_{\lambda_\ast}\}\cup\{0\}.$
        \item[(ii)] $t_{\lambda_\ast}$ is continuous and the map $P^-:S\cap\overline{\hat{\mathcal{M}}_{\lambda_\ast}}\to \mathcal{M}^-_{\lambda_\ast}\cup\mathcal{M}^0_{\lambda_\ast}$ defined by $P^-(u)=t_{\lambda_\ast}(u)u$ is a homeomorphism.
        \item[(iii)]  $s_{\lambda_\ast}$ is continuous and the map $P^+:S\to \mathcal{M}^+_{\lambda_\ast}\cup\mathcal{M}^0_{\lambda_\ast}$ defined by $P^+(u)=s_{\lambda_\ast}(u)u$ is a homeomorphism.
        \item[(iv)] $\mathcal{M}^0_{\lambda_\ast}$ has empty interior .
    \end{itemize}
\end{proposition}
Define $$\hat{\Psi}^-_{\lambda_\ast}=\inf\{\Psi_{\lambda_\ast}(t_{\lambda_\ast}(u)u):~u\in \mathcal{M}^-_{\lambda_\ast}\cup\mathcal{M}^0_{\lambda_\ast}\}$$

and $$\hat{\Psi}^+_{\lambda_\ast}=\inf\{\Psi_{\lambda_\ast}(s_{\lambda_\ast}(u)u):~u\in \mathcal{M}^+_{\lambda_\ast}\cup\mathcal{M}^0_{\lambda_\ast}\}.$$

Then from Corollary \ref{cor5.3} and Proposition \ref{prop5.4}, we notice that $\hat{\Psi}^\pm_{\lambda_\ast}=\Upsilon^\pm_{\lambda_\ast}$ (see \cite{silva2018local}).
\begin{proposition}\label{prop5.5}
   The maps $(0,\lambda_\ast]\ni \lambda\mapsto \Upsilon^\pm_{\lambda} $  are decreasing and  left continuous for $\lambda\in(0,\lambda_\ast)$ . Moreover,$$\lim_{\lambda\uparrow{\lambda_\ast}}\Upsilon^\pm_{\lambda}=\Upsilon^\pm_{\lambda_\ast}. $$
\end{proposition}
\begin{proof}
 Let $0<\lambda<\hat{\lambda}<\lambda_\ast$. By Lemma \ref{lem4.2} and Lemma \ref{lem3.7}, we have  $\mathcal{I}^-_\lambda(w_\lambda)=\Upsilon^-_{\lambda}$ and the map $\lambda\mapsto  \mathcal{I}^-_{\lambda}(u)$ is strictly decreasing. Using these informations, we have 
 $$ \Upsilon^-_{\hat{\lambda}}\leq \mathcal{I}^-_{\hat{\lambda}}(w_\lambda)< \mathcal{I}^-_\lambda(w_\lambda)=\Upsilon^-_{\lambda}.$$
 It follows that $\Upsilon^-_{\lambda}$ is strictly decreasing for $0<\lambda<\lambda_\ast$. If $0<\lambda<\lambda_\ast$, then by Corollary \ref{cor5.3} and Lemma  \ref{lem3.7}, we have
 $$\Upsilon^-_{\lambda_\ast}\leq\Psi_{\lambda_\ast}(t_{\lambda_\ast}(u)u)=\lim_{\lambda\uparrow {\lambda_\ast}}\mathcal{I}^-_{\lambda}(u)=\mathcal{I}^-_{\lambda_\ast}(u)< \mathcal{I}^-_{\lambda}(u).$$ This yields $\Upsilon^-_{\lambda_\ast}\leq \Upsilon^-_{\lambda}$ and hence $\Upsilon^-_{\lambda}$ is decreasing for $0<\lambda\leq\lambda_\ast$.
 To prove the left continuity of $\Upsilon^-_{\lambda}$, choose a sequence $\{\lambda_n\}_{n\in N}$ such that $\lambda_n\uparrow \lambda\in (0,\lambda_\ast)$ as $n\to\infty$. Since the map $(0,\lambda_\ast)\ni\lambda\mapsto  \Upsilon^-_{\lambda}$ is strictly decreasing and $\lambda_n<\lambda$ for $n$ large enough . It follows that $\Upsilon^-_{\lambda}<\Upsilon^-_{\lambda_n}$ for $n$ large enough and hence $\Upsilon^-_{\lambda}\leq\displaystyle\lim_{n\to\infty}\Upsilon^-_{\lambda_n}$. This shows that $$\Upsilon^-_{\lambda}\leq\lim_{n\to\infty}\Upsilon^-_{\lambda_n}\leq\lim_{n\to\infty}\mathcal{I}^-_{\lambda_n}(w_\lambda)=\lim_{n\to\infty}\Psi_{\lambda_n}(t^-_{\lambda_n}(w_\lambda)w_\lambda)=\Psi_{\lambda}(t^-_{\lambda}(w_\lambda)w_\lambda) =\mathcal{I}^-_\lambda(w_\lambda)=\Upsilon^-_{\lambda}.$$
 Hence $$\lim_{n\to\infty}\Upsilon^-_{\lambda_n}=\Upsilon^-_{\lambda},~\text{i.e.},~ \lim_{\lambda_n\uparrow \lambda}\Upsilon^-_{\lambda_n}=\Upsilon^-_{\lambda},~\text{for all}~\lambda\in (0,\lambda_\ast).$$
 Therefore, the map $\Upsilon^-_{\lambda}$ is left continuous for $\lambda\in(0,\lambda_\ast)$. Next, our aim is to prove $\Upsilon^-_{\lambda}$ is left continuous at $\lambda=\lambda_\ast$, that is, $\displaystyle\lim_{\lambda\uparrow{\lambda_\ast}}\Upsilon^-_{\lambda}=\Upsilon^-_{\lambda_\ast}.$ For this, take a $\{\lambda_n\}_{n\in N}$ such that $\lambda_n\uparrow \lambda_\ast$ as $n\to\infty$, therefore $\lambda_n<\lambda_\ast$ for $n$ large enough. The decreaseness of the map $(0,\lambda_\ast]\ni \lambda\mapsto \Upsilon^-_{\lambda} $ infer that $\Upsilon^-_{\lambda_\ast}\leq\Upsilon^-_{\lambda_n}$ for $n$ large enough and hence $\Upsilon^-_{\lambda_\ast}\leq\displaystyle\lim_{n\to\infty}\Upsilon^-_{\lambda_n}=\Upsilon$ (say). We claim that $\Upsilon=\Upsilon^-_{\lambda_\ast}$. Indeed, if not, then let there exist a $\rho>0$ such that $\Upsilon-\Upsilon^-_{\lambda_\ast}\geq \rho $. Choosing $\rho'>0$ with $2\rho'<\rho$ and  $w_{\rho'}\in \mathcal{M}^-_{\lambda_\ast}$ such that $\mathcal{I}^-_{\lambda_\ast}(w_{\rho'})\leq \Upsilon^-_{\lambda_\ast}+\rho'$. Now from the continuity of the map $\lambda\mapsto \mathcal{I}^-_{\lambda}(u)$, we obtain  $\mathcal{I}^-_{\lambda_n}(w_{\rho'})\to\mathcal{I}^-_{\lambda_\ast}(w_{\rho'}) $ as $\lambda_n\uparrow \lambda_\ast$. Thus for $n$ large enough, we have
 $$ 0\leq \mathcal{I}^-_{\lambda_n}(w_{\rho'})-\mathcal{I}^-_{\lambda_\ast}(w_{\rho'})\leq \rho'.$$
 It follows that $$\Upsilon^-_{\lambda_n} \leq \mathcal{I}^-_{\lambda_n}(w_{\rho'})\leq \mathcal{I}^-_{\lambda_\ast}(w_{\rho'})+\rho'\leq\Upsilon^-_{\lambda_\ast}+2\rho'\leq \Upsilon-\rho+2\rho'<\Upsilon.$$
Passing limit $n\to\infty$ in the above inequality, we have $\Upsilon\leq \Upsilon-\rho+2\rho'<\Upsilon$, which is a contradiction and hence the proof is completed.\\
Similarly, we can prove that all the above results hold for the map $\Upsilon^+_{\lambda}$. 
\end{proof}
\begin{theorem}\label{thm5.6}
 The problem $(\mathcal{E}_{\lambda_\ast})$ has at least two solutions  $w_{\lambda_\ast}\in\mathcal{M}^-_{\lambda_\ast}$ and $u_{\lambda_\ast}\in\mathcal{M}^+_{\lambda_\ast}$. Moreover, $$ \mathcal{I}^-_{\lambda_\ast}(w_{\lambda_\ast})=\Upsilon^-_{\lambda_\ast}~\text{and}~\mathcal{I}^+_{\lambda_\ast}(u_{\lambda_\ast})=\Upsilon^+_{\lambda_\ast}. $$
\end{theorem}
\begin{proof}
  To prove  $w_{\lambda_\ast}\in\mathcal{M}^-_{\lambda_\ast}$ is a solution for $(\mathcal{E}_{\lambda_\ast})$, take $\lambda_n\uparrow \lambda_\ast$ as $n\to\infty$. Let $\{u_n\}_{n\in \mathbb{N}}$ be a sequence in $\mathcal{M}^-_{\lambda_n}$ such that $\Upsilon^-_{\lambda_n}=\mathcal{I}^-_{\lambda_n}(u_n)$ and also $u_n$ be a solution to $(\mathcal{E}_{\lambda_n})$ for each $n\in \mathbb{N}$. Now we claim that the sequence  $\{u_n\}_{n\in \mathbb{N}}$ is bounded in $W^{s,p}_{V}(\mathbb{R}^N)$. Indeed, by similar to the proof of Lemma \ref{lem3.2}-$(ii)$ (see \eqref{eq3.16}), we can deduce that $\{u_n\}_{n\in \mathbb{N}}$ is bounded from below. Now we only have to prove $\{u_n\}_{n\in \mathbb{N}}$ is bounded from above. If not, let $\|u_n\|\to\infty$ as $n\to\infty$, then using the fact $u_n\in\mathcal{M}^-_{\lambda_n}$ and Proposition \ref{prop5.5}, we have
  $$\Upsilon^-_{\lambda_\ast}=\lim_{n\to\infty} \mathcal{I}^-_{\lambda_n}(u_n)=\lim_{n\to\infty} \Psi_{\lambda_n}(u_n)\geq a\bigg(\frac{1}{p(m+1)}-\frac{1}{\gamma}\bigg)\|u_n\|^{p(m+1)}-c{\lambda_n}\|\alpha\|_{L^\xi(\mathbb{R}^N)}\bigg(\frac{1}{1-\delta}-\frac{1}{\gamma}\bigg)\|u_n\|^{1-\delta} $$
  $$\hspace{3cm} \to\infty~\text{as}~n \to\infty,~\text{since}~0<1-\delta<1<p(m+1)<\gamma,$$
  which is a contradiction. It follows that $\{u_n\}_{n\in \mathbb{N}}$ must be bounded and consequently, up to a subsequence $u_n\rightharpoonup w_{\lambda_\ast}$ weakly in $W^{s,p}_{V}(\mathbb{R}^N)$ as $n\to\infty$. By Lemma \ref{lemma2.3}, we have $u_n\to w_{\lambda_\ast}$ in $L^\tau(\mathbb{R}^N)$ and $L^\gamma(\mathbb{R}^N)$ respectively as $n\to\infty$, $u_n\to w_{\lambda_\ast}$ a.e. $x,y\in\mathbb{R}^N$ and $w_{\lambda_\ast}\geq 0$. Note that \eqref{eq4.2} still holds with $u_\lambda$ replaced by $w_{\lambda_\ast}$. 
    The boundedness of $\{u_n\}_{n\in \mathbb{N}}$ in $W^{s,p}_{V}(\mathbb{R}^N)$ infer that $\bigg\{\frac{|u_n(x)-u_n(y)|^{p-2}(u_n(x)-u_n(y))}{|x-y|^{(N+sp)/p^{\prime}}}\bigg\}_{n\in \mathbb{N}}  \text{and } \big\{V^{{\frac{1}{p^{\prime}}}}|u_n|^{p-2}u_n\big\}_{n\in \mathbb{N}} $ are bounded in $L^{p^{\prime}}(\mathbb{R}^{2N})$ and $L^{p^{\prime}}(\mathbb{R}^{N})$ respectively, where $p^{\prime}=\frac{p}{p-1}$ is the conjugate exponent of $p$. Moreover, we have
    $$ \frac{|u_n(x)-u_n(y)|^{p-2}(u_n(x)-u_n(y))}{|x-y|^{(N+sp)/p^{\prime}}}\to \frac{|w_{\lambda_\ast}(x)-w_{\lambda_\ast}(y)|^{p-2}(w_{\lambda_\ast}(x)-w_{\lambda_\ast}(y))}{|x-y|^{(N+sp)/p^{\prime}}}~\text{a.e. in }~\mathbb{R}^{2N}$$
    and $$V^{\frac{1}{p^{\prime}}}|u_n|^{p-2}u_n\to V^{\frac{1}{p^{\prime}}}|w_{\lambda_\ast}|^{p-2}w_{\lambda_\ast}~\text{a.e. in }~\mathbb{R}^{N}~\text{as}~n\to\infty .$$
    It follows that $$ \frac{|u_n(x)-u_n(y)|^{p-2}(u_n(x)-u_n(y))}{|x-y|^{(N+sp)/p^{\prime}}}\rightharpoonup \frac{|w_{\lambda_\ast}(x)-w_{\lambda_\ast}(y)|^{p-2}(w_{\lambda_\ast}(x)-w_{\lambda_\ast}(y))}{|x-y|^{(N+sp)/p^{\prime}}}~\text{weakly in }~L^{p^{\prime}}(\mathbb{R}^{2N})$$
    and $$V^{\frac{1}{p^{\prime}}}|u_n|^{p-2}u_n\rightharpoonup V^{\frac{1}{p^{\prime}}}|w_{\lambda_\ast}|^{p-2}w_{\lambda_\ast}~\text{weakly in }~L^{p^{\prime}}(\mathbb{R}^{N})~\text{as}~n\to\infty .$$
    Due to  weak convergence in Lebesgue space, for any $v\in W^{s,p}_{V}(\mathbb{R}^N) $, we obtain
    \begin{equation}\label{eq5.10}
\lim_{n\to\infty}\langle{B(u_n),v}\rangle=\langle{B(w_{\lambda_\ast}),v}\rangle .       
 \end{equation}
Since $u_n\to w_{\lambda_\ast}$ in $L^\gamma(\mathbb{R}^N)$  as $n\to\infty$, therefore by applying Lebesgue dominated convergence theorem, we have $|u_n|^{\gamma-2}u_n\to|w_{\lambda_\ast}|^{\gamma-2}w_{\lambda_\ast} $ in $L^{\frac{\gamma}{\gamma-1}}(\mathbb{R}^{N})$ as $n\to\infty$. Consequently, by Holder's inequality, we have 
$$\bigg|\int_{\mathbb{R}^{N}}\beta(x)\big(|u_n|^{\gamma-2}u_n-|w_{\lambda_\ast}|^{\gamma-2}w_{\lambda_\ast}\big)v~\mathrm{d}x\bigg|\leq \|\beta(x)\|_{L^\infty(\mathbb{R}^{N})}\bigg\||u_n|^{\gamma-2}u_n-|w_{\lambda_\ast}|^{\gamma-2}w_{\lambda_\ast}\bigg\|_{L^{\frac{\gamma}{\gamma-1}}(\mathbb{R}^{N})}\|v\|_{L^\gamma(\mathbb{R}^{N})}$$ $$
\to 0~\text{as}~n\to\infty .$$
It follows that 
\begin{equation}\label{eq5.11}
 \lim_{n\to\infty}\int_{\mathbb{R}^{N}}\beta(x)|u_n|^{\gamma-2}u_n v~\mathrm{d}x=\int_{\mathbb{R}^{N}}\beta(x)|w_{\lambda_\ast}|^{\gamma-2}w_{\lambda_\ast}v~\mathrm{d}x.   
\end{equation}
Using the fact that $u_n$ is a solution of the problem $(\mathcal{E}_{\lambda_n})$ for each $n\in \mathbb{N}$, we have 
\begin{equation}\label{eq5.12}   
\langle{L(u_n),v}\rangle=\lambda_n \int_{\mathbb{R}^{N}}\alpha(x)u^{-\delta}_{n}v~\mathrm{d}x+\int_{\mathbb{R}^{N}}\beta(x)u_n^{\gamma-1} v~\mathrm{d}x.    
\end{equation}
Next, applying limit inferior on both sides of \eqref{eq5.12} as $n\to\infty$ with $v\in\mathcal{K}$ and using the fact that \eqref{eq4.2}, \eqref{eq5.10} and \eqref{eq5.11} hold, we get
\begin{equation}\label{eq5.13}   
\infty >\liminf_{n\to\infty}~ \langle{L(u_n),v}\rangle-\int_{\mathbb{R}^{N}}\beta(x)w_{\lambda_\ast}^{\gamma-1} v~\mathrm{d}x\geq\lambda_\ast\liminf_{n\to\infty} \int_{\mathbb{R}^{N}}\alpha(x)u^{-\delta}_{n}v~\mathrm{d}x.    
\end{equation}
Furthermore, by Fatou's lemma, we have
\begin{equation}\label{eq5.14}
  \liminf_{n\to\infty} \int_{\mathbb{R}^{N}}\alpha(x)u^{-\delta}_{n}v~\mathrm{d}x\geq  \int_{\mathbb{R}^{N}}H(x)~\mathrm{d}x,
\end{equation}
where $H(x)$ is defined as in \eqref{eq5.999}, with $u$ replaced by $w_{\lambda_\ast}$.
From \eqref{eq5.13} and \eqref{eq5.14}, we obtain $0\leq \displaystyle\int_{\mathbb{R}^{N}}H(x)~\mathrm{d}x<\infty$ and hence $H(x)=\alpha(x)w^{-\delta}_{\lambda_\ast}v$, i.e., $w_{\lambda_\ast}>0$ a.e. in $ \mathbb{R}^{N}$ and $\alpha(x)w^{-\delta}_{\lambda_\ast}v\in L^1(\mathbb{R}^{N})$.
Consequently, from \eqref{eq4.2} and \eqref{eq5.11}-- \eqref{eq5.14}, we have $$\limsup_{n\to\infty} \langle{L(u_n)-L(~w_{\lambda_\ast}),u_n-~w_{\lambda_\ast}}\rangle=\limsup_{n\to\infty} \langle{L(u_n),u_n-~w_{\lambda_\ast}}\rangle\leq\limsup_{n\to\infty} \langle{L(u_n),u_n}\rangle-\liminf_{n\to\infty} \langle{L(u_n),w_{\lambda_\ast}}\rangle $$
$$\hspace{-5cm}=\limsup_{n\to\infty}M(\|u_n\|^p)\langle{B(u_n),u_n}\rangle-\liminf_{n\to\infty}M(\|u_n\|^p)\langle{B(u_n),w_{\lambda_\ast}}\rangle $$
$$\hspace{1cm}=\limsup_{n\to\infty}\bigg[\lambda_n \int_{\mathbb{R}^{N}}\alpha(x)u^{1-\delta}_{n}~\mathrm{d}x+\int_{\mathbb{R}^{N}}\beta(x)u_n^{\gamma} ~\mathrm{d}x\bigg]-\liminf_{n\to\infty}\bigg[\lambda_n \int_{\mathbb{R}^{N}}\alpha(x)u^{-\delta}_{n}w_{\lambda_\ast}~\mathrm{d}x+\int_{\mathbb{R}^{N}}\beta(x)u_n^{\gamma-1} w_{\lambda_\ast}~\mathrm{d}x\bigg] $$
$$\hspace{-1.2cm}\leq \bigg(\lambda_\ast \int_{\mathbb{R}^{N}}\alpha(x)w^{1-\delta}_{\lambda_\ast}~\mathrm{d}x+\int_{\mathbb{R}^{N}}\beta(x)w_{\lambda_\ast}^{\gamma} ~\mathrm{d}x\bigg)-\lambda_\ast \int_{\mathbb{R}^{N}}\alpha(x)w^{1-\delta}_{\lambda_\ast}~\mathrm{d}x-\int_{\mathbb{R}^{N}}\beta(x)w_{\lambda_\ast}^{\gamma} ~\mathrm{d}x=0. $$
Thus, by Lemma \ref{lemma2.4}, we ge $u_n\to w_{\lambda_\ast}$ in $W^{s,p}_{V}(\mathbb{R}^N)$ as $n\to\infty$ and hence $ M\big(\|u_n\|^p\big)\to M\big(\|w_{\lambda_\ast}\|^p\big)$ in $\mathbb{R}$ as $n\to\infty$. Now applying a limit inferior to \eqref{eq5.12} and using Fatou's lemma, we obtain 
\begin{equation}\label{eq5.15}
 \langle{L(w_{\lambda_\ast}),v}\rangle-\int_{\mathbb{R}^{N}}\beta(x)w_{\lambda_\ast}^{\gamma-1} v~\mathrm{d}x\geq\lambda_\ast \int_{\mathbb{R}^{N}}\alpha(x)w^{-\delta}_{\lambda_\ast}v~\mathrm{d}x,~\forall~v\in\mathcal{K}.   
\end{equation}
In addition to this, we also have
$$\mathcal{F}'_{\lambda_\ast,w_{\lambda_\ast}}(1)=0~\text{and}~\mathcal{F}''_{\lambda_\ast,w_{\lambda_\ast}}(1)\leq 0~\text{and hence} \int_{\mathbb{R}^{N}}\beta(x)|w_{\lambda_\ast}|^{\gamma}~\mathrm{d}x >0.$$ It follows that $w_{\lambda_\ast}\in \mathcal{M}^-_{\lambda_\ast}\cup\mathcal{M}^0_{\lambda_\ast}$ and therefore, we obtain
 $$ M\big(\|w_{\lambda_\ast}\|^p\big)\|w_{\lambda_\ast}\|^p-\lambda_\ast \int_{\mathbb{R}^{N}} \alpha(x)w_{\lambda_\ast}^{1-\delta}~\mathrm{d}x-\int_{\mathbb{R}^{N}}\beta(x)w_{\lambda_\ast}^{\gamma}~\mathrm{d}x =0.$$   
Let $v\in W^{s,p}_{V}(\mathbb{R}^N)$ and define $\phi_\epsilon= w_{\lambda_\ast}+\epsilon v$, then for each $\epsilon>0$ be given $\phi^+_\epsilon\in\mathcal{K}$. Now replacing $\phi_\epsilon$ with $v$ in \eqref{eq5.15} and adapting arguments similar to those of Theorem \ref{thm4.6}, we can show that $w_{\lambda_\ast}$ is a solution of $(\mathcal{E}_{\lambda_\ast})$. By Lemma \ref{lem5.1}, we get $w_{\lambda_\ast}\notin \mathcal{M}^0_{\lambda_\ast} $ and hence $w_{\lambda_\ast}\in \mathcal{M}^-_{\lambda_\ast} $. Moreover, we deduce from strong convergence and Proposition \ref{prop5.5} that $$ \mathcal{I}^-_{\lambda_\ast}(w_{\lambda_\ast})=\Psi_{\lambda_\ast}(w_{\lambda_\ast})=\lim_{n\to\infty} \Psi_{\lambda_n}(u_n)=\lim_{n\to\infty} \mathcal{I}^-_{\lambda_n}(u_n)=\Upsilon^-_{\lambda_\ast}.$$
Similarly, we can prove that $u_{\lambda_\ast}\in\mathcal{M}^+_{\lambda_\ast}$ is a solution for $(\mathcal{E}_{\lambda_\ast})$ and $\mathcal{I}^+_{\lambda_\ast}(u_{\lambda_\ast})=\Upsilon^+_{\lambda_\ast}$. Hence, the theorem is well established.
\end{proof}
For $0<\lambda\leq\lambda_\ast$, define the sets $$\mathcal{S}^+_{\lambda}=\big\{ u\in \mathcal{M}^+_{\lambda}:~\mathcal{I}^+_{\lambda}(u)=\Upsilon^+_{\lambda}\big\}~\text{and}~ \mathcal{S}^-_{\lambda}=\big\{ w\in \mathcal{M}^-_{\lambda}:~\mathcal{I}^-_{\lambda}(w)=\Upsilon^-_{\lambda}\big\}.$$
The following Corollary follows directly from Lemma \ref{lem4.1}, Lemma \ref{lem4.2} and Theorem \ref{thm5.6}. 
\begin{corollary}\label{cor5.7}
    The following holds for $0<\lambda\leq\lambda_\ast$: 
    \begin{itemize}
        \item [(i)] $\mathcal{S}^+_{\lambda}$ and $\mathcal{S}^-_{\lambda}$ are non-empty, compact and hence  there exist $c_\lambda,C_\lambda>0$  such that $c_\lambda\leq \|u\|,\|w\|\leq C_\lambda $ for all $u\in \mathcal{S}^+_{\lambda}$ and $w\in \mathcal{S}^-_{\lambda}$. 
        \item[(ii)] if $u\in \mathcal{S}^+_{\lambda}\cup \mathcal{S}^-_{\lambda}$, then $u$ is a solution for \eqref{main problem}.
    \end{itemize}
\end{corollary}
\section{Existence of solutions for \texorpdfstring{$\lambda>\lambda_\ast$}{Lg}}\label{sec6}
In this section, we study the existence of solutions for the problem \eqref{main problem} when $\lambda$ crosses the extremal parameter $\lambda_\ast$. The aim is to explore the minimization problem on appropriate subsets of $\mathcal{M}^+_{\lambda_\ast}$ and $\mathcal{M}^-_{\lambda_\ast}$ that preserve an acceptable distance from the set $\mathcal{M}^0_{\lambda_\ast}$ and minimizers achieved on these sets may be projected on $\hat{\mathcal{M}}_{\lambda}$ and $\hat{\mathcal{M}}_{\lambda}\cup \hat{\mathcal{M}}^+_{\lambda}$ for $\lambda\in (\lambda_\ast,\lambda_\ast+\epsilon)$, where $\epsilon>0$ small enough.\\
Let $\lambda>0$, then for all $w\in\mathcal{M}^{\mp}_\lambda\cup\mathcal{M}^0_\lambda$, we define
$$\mathcal{J}^\mp_\lambda(w)=a(p(m+1)-1)\|w\|^{p(m+1)}+\lambda\delta \int_{\mathbb{R}^{N}}\alpha(x)|w|^{1-\delta}~\mathrm{d}x-(\gamma-1)\int_{\mathbb{R}^{N}}\beta (x)|w|^\gamma~\mathrm{d}x.$$
\begin{lemma}\label{lem6.1}
 Suppose that, $0<\hat{C}^+_1<\hat{C}^-_2$ and $\lambda_n\downarrow \lambda_\ast$ as $n\to\infty$. Also, assume $u_n\in \mathcal{M}^{\mp}_{\lambda_\ast}$ such that $\hat{C}^+_1\leq\|u_n\|\leq \hat{C}^-_2$ for each $n\in\mathbb{N}$ and $\mathcal{J}^{\mp}_{\lambda_n}(t^{\mp}_{\lambda_n}(u_n)u_n)\to 0$ as $n\to\infty$, then $\mathrm{dist} (u_n,\mathcal{M}^0_{\lambda_\ast})\to 0$ as $n\to\infty$.
\end{lemma}
\begin{proof}
     We prove for $u_n\in \mathcal{M}^-_{\lambda_\ast}$. Since $u_n\in \mathcal{M}^-_{\lambda_\ast}$ by Lemma \ref{lem3.2}-$(ii)$ (see \eqref{eq3.14}) that there exists $c_1>0$ such that $\displaystyle\int_{\mathbb{R}^{N}}\beta (x)|u_n|^\gamma~\mathrm{d}x>c_1$. Further, as $t^-_{\lambda_n}(u_n)u_n\in\mathcal{M}^-_{\lambda_n}$, hence by Proposition \ref{prop3.1}, there exists $t^+_{\lambda_n}(u_n)>0$ satisfying $t^+_{\lambda_n}(u_n)<t^-_{\lambda_n}(u_n)$ such that $t^+_{\lambda_n}(u_n)u_n\in\mathcal{M}^+_{\lambda_n}$. Setting $t^+_n=t^+_{\lambda_n}(u_n)$ and $t^-_n=t^-_{\lambda_n}(u_n)$, we have
        \begin{equation}\label{eq6.1}
          \mathcal{F}'_{\lambda_n,u_n}(t^-_n)= \mathcal{F}'_{\lambda_n,u_n}(t^+_n)=0~\text{and}~ \mathcal{J}^-_{\lambda_n}(t^-_n u_n)=o(1)~\text{as}~n\to\infty. 
        \end{equation}
        In solving these equations, we obtain
        \begin{equation}\label{eq6.2}
        \begin{split}
          a t^-_n\|u_n\|^{p(m+1)}\Bigg[\frac{\big(p(m+1)+\delta-1\big)\big(\frac{t^+_n}{t^-_n}\big)^{\gamma+\delta-1}-(\gamma+\delta-1)\big(\frac{t^+_n}{t^-_n}\big)^{p(m+1)+\delta-1}+\big(\gamma-p(m+1)\big)}{\big(\frac{t^+_n}{t^-_n}\big)^{\gamma+\delta-1}-1}\Bigg]=o(1) 
        \end{split}  
        \end{equation} as $n\to\infty$.
        Using $\hat{C}^+_1\leq\|u_n\|\leq \hat{C}^-_2$ for each $n\in\mathbb{N}$ and Lemma \ref{lem3.2}, we deduce that $t^-_n$ and $t^+_n$ are bounded. Hence, up to subsequences $ t^-_n\to\sigma$ and $ t^+_n\to\eta$ as $n\to\infty$. Now from \eqref{eq6.2}, we obtain
        $$\big(p(m+1)+\delta-1\big)\bigg(\frac{\eta}{\sigma}\bigg)^{\gamma+\delta-1}-(\gamma+\delta-1)\bigg(\frac{\eta}{\sigma}\bigg)^{p(m+1)+\delta-1}+\big(\gamma-p(m+1)\big)=0,$$
        which has a unique root $\eta=\sigma$ and hence $ t^-_n\to\eta$ and $ t^+_n\to\eta$ as $n\to\infty$. Also, since $t^+_n u_n\in\mathcal{M}^+_{\lambda_n}$, there exists $c_2>0$ such that $\displaystyle\int_{\mathbb{R}^{N}}\alpha(x)|u_n|^{1-\delta}~\mathrm{d}x>c_2$. From \eqref{eq6.1}, it follows that         
         $$\mathcal{F}'_{\lambda_\ast,u_n}(\eta)=o(1)~\text{and}~\mathcal{J}^-_{\lambda_\ast}(\eta u_n)=o(1)~\text{as}~n\to\infty.$$          
This yields $$\frac{a\big(\gamma-p(m+1)\big)\|\eta u_n\|^{p
(m+1)}}{(\delta+\gamma-1)\displaystyle\int_{\mathbb{R}^{N}}\alpha(x)|\eta u_n|^{1-\delta}~\mathrm{d}x}=\lambda_\ast+o(1) ~\text{as}~n\to\infty$$
and $$  \frac{a\big(p(m+1)+\delta-1\big)\|\eta u_n\|^{p
(m+1)}}{(\delta+\gamma-1)\displaystyle\int_{\mathbb{R}^{N}}\beta(x)|\eta u_n|^{\gamma}~\mathrm{d}x}=1+ o(1)~\text{as}~n\to\infty.$$
Hence it follows from \eqref{eq3.18} and Lemma \ref{lem3.5} that
$$\lambda(u_n)=\lambda(\eta u_n)=(\lambda_\ast+o(1))(1+ o(1))^{\frac{p(m+1)+\delta-1}{\gamma-p(m+1)}}\to \lambda_\ast~\text{as}~n\to\infty. $$
This shows that $\{u_n\}_{n\in\mathbb{N}}$ is a bounded minimizing sequence for $\lambda_\ast$. Therefore, up to a subsequence $u_n\rightharpoonup u$ weakly in $W^{s,p}_{V}(\mathbb{R}^N)$ as $n\to\infty$. Repeating a similar procedure done in Corollary \ref{cor5.2}, we get $u_n\to u$ in $W^{s,p}_{V}(\mathbb{R}^N)$ as $n\to\infty$. Now, from the continuity of $\lambda(u)$, we infer that $\lambda(u)=\lambda_\ast$ and also that we have $\displaystyle\int_{\mathbb{R}^{N}}\beta (x)|u|^\gamma~\mathrm{d}x>0$. It follows from strong convergence that $\mathcal{F}'_{\lambda_\ast,u}(1)=\displaystyle\lim_{n\to\infty} \mathcal{F}'_{\lambda_n,u_n}(1)=0,~\text{that is},~u\in\mathcal{M}_{\lambda_\ast}$ and hence by Corollary \ref{cor3.6}-$(c)$, we deduce that $u\in\mathcal{M}^0_{\lambda_\ast}$ . This implies that 
$$\lim_{n\to\infty} \mathrm{dist}(u_n,\mathcal{M}^0_{\lambda_\ast})=\mathrm{dist}(u,\mathcal{M}^0_{\lambda_\ast})= 0.$$
By arguing similarly as above, we can prove for the case $u_n\in \mathcal{M}^-_{\lambda_\ast}$. This completes the proof. 
\end{proof}
Let $\hat{C}^+_1,\hat{C}^-_2,d>0$ and define the sets 
$$\mathcal{M}^-_{\lambda_\ast,\mathrm{d}}=\big\{w\in \mathcal{M}^-_{\lambda_\ast}:~\mathrm{dist}(w,\mathcal{M}^0_{\lambda_\ast})>d,~\|w\|\leq \hat{C}^-_2\big\} $$
and 
$$\mathcal{M}^+_{\lambda_\ast,\mathrm{d}}=\big\{u\in \mathcal{M}^+_{\lambda_\ast}:~\mathrm{dist}(u,\mathcal{M}^0_{\lambda_\ast})>d,~\hat{C}^+_1\leq\|u\|\big\}.$$
Lemma \ref{lem6.1} can immediately obtain the following Corollary.
\begin{corollary}\label{cor6.2}
 Let $\hat{C}^+_1, \hat{C}^-_2,d>0$ be given as above, then there exists $\epsilon>0$ such that
 \begin{itemize}
     \item [(i)] there exists $\eta<0$ such that $\mathcal{J}^-_{\lambda}(t^-_{\lambda}(w)w)< \eta$ for all $\lambda\in(\lambda_\ast,\lambda_\ast+\epsilon)$ and $w\in\mathcal{M}^-_{\lambda_\ast,\mathrm{d}}$. Consequently, $t^-_{\lambda}(w)w\in\mathcal{M}^-_\lambda$ and $w\in\hat{\mathcal{M}}_\lambda$, for all $\lambda\in(\lambda_\ast,\lambda_\ast+\epsilon)$.
     \item [(ii)] there exists $\eta>0$ such that $\mathcal{J}^+_{\lambda}(t^+_{\lambda}(u)u)> \eta$ for all $\lambda\in(\lambda_\ast,\lambda_\ast+\epsilon)$ and $u\in\mathcal{M}^+_{\lambda_\ast,\mathrm{d}}$. Consequently, $t^+_{\lambda}(u)u\in\mathcal{M}^+_\lambda$ and $u\in\hat{\mathcal{M}}_\lambda\cup\hat{\mathcal{M}}^+_\lambda $, for all $\lambda\in(\lambda_\ast,\lambda_\ast+\epsilon)$.
 \end{itemize}
\end{corollary}
\begin{lemma}\label{lem6.3}
 There holds $\mathrm{dist}\big(\mathcal{S}^\pm_{\lambda_\ast},\mathcal{M}^0_{\lambda_\ast}\big)>0.$ 
\end{lemma}
\begin{proof}
First, we show that $\mathrm{dist}\big(\mathcal{S}^-_{\lambda_\ast},\mathcal{M}^0_{\lambda_\ast}\big)>0$. Indeed, if not, then $\mathrm{dist}\big(\mathcal{S}^-_{\lambda_\ast},\mathcal{M}^0_{\lambda_\ast}\big)=0$. So, there exist two sequences $\{w_n\}_{n\in\mathbb{N}}\subset \mathcal{S}^-_{\lambda_\ast} $ and $\{\psi_n\}_{n\in\mathbb{N}}\subset\mathcal{M}^0_{\lambda_\ast}$ such that $\|w_n-\psi_n\|\to 0$ as $n\to\infty$. As $w_n\in\mathcal{S}^-_{\lambda_\ast}$, therefore by Corollary \ref{cor5.7}, we conclude that $w_n$ is a solution of $(\mathcal{E}_{\lambda_\ast})$. Therefore, we have    
$$\langle{L(w_n),v}\rangle-\lambda_\ast \int_{\mathbb{R}^{N}}\alpha(x)w^{-\delta}_n v~\mathrm{d}x-\int_{\mathbb{R}^{N}}\beta(x)w_n^{\gamma-1} v~\mathrm{d}x=0,~\text{for all}~v\in W^{s,p}_{V}(\mathbb{R}^N).$$
From Corollary \ref{cor5.2}, we deduce that there exists $\psi\in\mathcal{M}^0_{\lambda_\ast}$ such that up to a subsequence $\psi_n\to\psi$ in $ W^{s,p}_{V}(\mathbb{R}^N)$ as $n\to\infty$ and hence $w_n\to\psi$ in $ W^{s,p}_{V}(\mathbb{R}^N)$ as $n\to\infty$. Now, following the proof of Theorem \ref{thm5.6}, we obtain the following result.
\begin{equation}\label{eq6.5}   
\langle{L(\psi),v}\rangle-\lambda_\ast \int_{\mathbb{R}^{N}}\alpha(x)\psi^{-\delta} v~\mathrm{d}x-\int_{\mathbb{R}^{N}}\beta(x)\psi^{\gamma-1} v~\mathrm{d}x\geq 0,~\text{for all}~v\in\mathcal{K}.
\end{equation}
Let $v\in W^{s,p}_{V}(\mathbb{R}^N)$ and define $\phi_\epsilon= \psi+\epsilon v$, then for each $\epsilon>0$, $\phi^+_\epsilon\in\mathcal{K}$. Now replacing $\phi_\epsilon$ with $v$ in \eqref{eq6.5} and following Theorem \ref{thm4.6}, we can easily show that $\phi\in\mathcal{M}^0_{\lambda_\ast}$ is a solution to $(\mathcal{E}_{\lambda_\ast})$, which is a contradiction by Lemma \ref{lem5.1}. This shows that $\mathrm{dist}\big(\mathcal{S}^-_{\lambda_\ast},\mathcal{M}^0_{\lambda_\ast}\big)>0$.\\ Similarly, we can prove $\mathrm{dist}\big(\mathcal{S}^+_{\lambda_\ast},\mathcal{M}^0_{\lambda_\ast}\big)>0$. This completes the proof.
\end{proof}

Define $\mathrm{d}_{\lambda_\ast,\pm}=\mathrm{dist}\big(\mathcal{S}^\pm_{\lambda_\ast},\mathcal{M}^0_{\lambda_\ast}\big)$. Let there exist constants $\hat{C}^+_{1,\lambda_\ast},\hat{C}^-_{2,\lambda_\ast}>0$ such that $\|w\|\leq\hat{C}^-_{2,\lambda_\ast},~\forall~w\in \mathcal{S}^-_{\lambda_\ast}$ and $\hat{C}^+_{1,\lambda_\ast}\leq \|u\|,~\forall~u\in \mathcal{S}^+_{\lambda_\ast}$. Further, let $\mathrm{d}_{\pm}\in(0,\mathrm{d}_{\lambda_\ast,\pm})$, $\hat{C}^+_{1}<\hat{C}^+_{1,\lambda_\ast}$, $\hat{C}^-_{2,\lambda_\ast}<\hat{C}^-_{2}$ and $\lambda\in(\lambda_\ast,\lambda_\ast+\epsilon)$, where $\epsilon>0$ as in Corollary \ref{cor6.2}. Now consider the following constrained minimization problems:
\begin{equation}\label{eq6.6}
\Upsilon^-_{\lambda,d_{-}}=\inf\big\{\mathcal{I}^-_{\lambda}(w):w\in\mathcal{M}^-_{\lambda_\ast,d_{-}}\big\}~\text{and}~\Upsilon^+_{\lambda,d_{+}}=\inf\big\{\mathcal{I}^+_{\lambda}(u):u\in\mathcal{M}^+_{\lambda_\ast,d_{+}}\big\}.
\end{equation}
\begin{remark}
   It is also true that $\mathcal{S}^\pm_{\lambda_\ast}\subset \mathcal{M}^\pm_{\lambda_\ast,d_{\pm}}.$
\end{remark}
\begin{proposition}\label{prop6.5}
    The maps $\lambda\ni(\lambda_\ast,\lambda_\ast+\epsilon)\mapsto \Upsilon^\pm_{\lambda,d_{\pm}} $ are decreasing and there holds $$\lim_{\lambda\downarrow\lambda_\ast} \Upsilon^\pm_{\lambda,d_{\pm}}=\Upsilon^\pm_{\lambda_\ast}.$$
    
\end{proposition}
\begin{proof}
 Let $w\in \mathcal{M}^-_{\lambda_\ast,d_{-}}$, then by Lemma \ref{lem3.7}, we obtain $\Upsilon^-_{\lambda,d_{-}}\leq \mathcal{I}^-_{\lambda}(w)<\mathcal{I}^-_{\lambda'}(w)$
 for $\lambda_\ast<\lambda'<\lambda<\lambda_\ast+\epsilon$. It follows that $\Upsilon^-_{\lambda,d_{-}}\leq\Upsilon^-_{\lambda',d_{-}}$ and hence the map $\lambda\mapsto  \Upsilon^-_{\lambda,d_{-}} $ is decreasing for all $\lambda\in(\lambda_\ast,\lambda_\ast+\epsilon)$. Further, if $w\in \mathcal{S}^-_{\lambda_\ast} $, then it is easy to see that $\Upsilon^-_{\lambda,d_{-}}<\Upsilon^-_{\lambda_\ast}$ for all $\lambda\in(\lambda_\ast,\lambda_\ast+\epsilon)$. Now we claim that $\displaystyle\lim_{\lambda\downarrow\lambda_\ast} \Upsilon^-_{\lambda,d_{-}}=\Upsilon^-_{\lambda_\ast}.$ Indeed, if not, then let $\lambda_n\downarrow\lambda_\ast$ as $n\to\infty$ such that $\displaystyle\lim_{\lambda_n\downarrow\lambda_\ast} \Upsilon^-_{\lambda_n,d_{-}}=\Upsilon^-~(\text{say})<\Upsilon^-_{\lambda_\ast}$ as $n\to\infty$ . By using \eqref{eq6.6}, we get a sequence $\{w_n\}_{n\in\mathbb{N}}\subset\mathcal{M}^-_{\lambda_\ast,d_{-}} $ such that 
 \begin{equation}\label{eq6.7}
     |\mathcal{I}^-_{\lambda_n}(w_n)-\Upsilon^-_{\lambda_n,d_{-}}|\to 0~\text{as}~n\to\infty.
 \end{equation}
 Next, from the definition of $\mathcal{M}^{-}_{\lambda_\ast,d_{-}}$ and Corollary \ref{cor6.2}, we get $\{w_n\}_{n\in\mathbb{N}}$ is bounded and $t^-_{\lambda_n}(w_n)w_n\in\mathcal{M}^-_{\lambda_n}$ for $n$ large enough. Hence up to a subsequence $w_n\rightharpoonup w$ weakly in $W^{s,p}_{V}(\mathbb{R}^N)$ as $n\to\infty$. Lemma \ref{lemma2.3} says that $w_n\to w$ in $L^\tau(\mathbb{R}^N)$ and $L^\gamma(\mathbb{R}^N)$ respectively as $n\to\infty$, $w_n\to w$ a.e. in $\mathbb{R}^N$ and $w\geq 0$. Following the proof of Lemma \ref{lem4.2}, we get $w\neq 0$ ( since $ w_n\in \mathcal{M}^-_{\lambda_\ast}$ ) and $\displaystyle\int_{\mathbb{R}^N}\beta(x)|w|^\gamma~\mathrm{d}x>0$. Further, we claim that $w_n\to w$ in $W^{s,p}_{V}(\mathbb{R}^N)$ as $n\to\infty$. Indeed, if not, then we have
   $$\liminf_{n\to\infty}\mathcal{F}'_{\lambda_n,w_n}(t_{\lambda_\ast}(w))> \mathcal{F}'_{\lambda_\ast,w}(t_{\lambda_\ast}(w))=0.$$  
 It follows that $\mathcal{F}'_{\lambda_n,w_n}(t_{\lambda_\ast}(w)>0$ for $n$ large enough. Since $t^-_{\lambda_n}(w_n)w_n\in\mathcal{M}^-_{\lambda_n}$ for $n$ large enough, therefore by Proposition \ref{prop3.1} the map $\mathcal{F}_{\lambda_n,w_n}$ is strictly increasing in $(t^+_{\lambda_n}(w_n),t^-_{\lambda_n}(w_n))$ and hence $t^+_{\lambda_n}(w_n)<t_{\lambda_\ast}(w)<t^-_{\lambda_n}(w_n)$ for $n$ large enough. Therefore, we have 
 $$ \Psi_{\lambda_\ast}(t_{\lambda_\ast}(w)w)<\liminf_{n\to\infty}\Psi_{\lambda_n}(t_{\lambda_\ast}(w)w_n)=\liminf_{n\to\infty}\mathcal{F}_{\lambda_n,w_n}(t_{\lambda_\ast}(w))< \liminf_{n\to\infty}\Psi_{\lambda_n,}(t^-_{\lambda_n}(w_n)w_n)$$ $$=\liminf_{n\to\infty}\mathcal{I}^-_{\lambda_n}(w_n)=\liminf_{n\to\infty}\Upsilon^-_{\lambda_n,d_{-}}=\Upsilon^-<\Upsilon^-_{\lambda_\ast},$$
 which is a contradiction because if we take $\lambda'_n\uparrow\lambda_\ast$ as $n\to\infty$, then by Corollary \ref{cor5.3} and Proposition \ref{prop5.5}, we get
 $$\Upsilon^-_{\lambda_\ast}=\lim_{\lambda'_n\uparrow{\lambda_\ast}}\Upsilon^-_{\lambda'_n}\leq \lim_{\lambda'_n\uparrow{\lambda_\ast}}\Psi_{\lambda'_n}(t^-_{\lambda'_n}(w)w)=\Psi_{\lambda_\ast}(t_{\lambda_\ast}(w)w). $$
 This shows that $w_n\to w$ in $W^{s,p}_{V}(\mathbb{R}^N)$ as $n\to\infty$. Next, by continuity of the function $\lambda\ni(\lambda_\ast,\lambda_\ast+\epsilon)\mapsto t^-_\lambda(w)$ (see Lemma \ref{lem3.7}-$(i)$), we get 
 \begin{equation}\label{eq6.9}
     |\mathcal{I}^-_{\lambda_\ast}(w_n)-\mathcal{I}^-_{\lambda_n}(w_n)|=|\Psi_{\lambda_\ast}(t^-_{\lambda_\ast}(w_n)w_n)-\Psi_{\lambda_n}(t^-_{\lambda_n}(w_n)w_n)|\to 0~\text{as}~n\to\infty.
 \end{equation}
 From \eqref{eq6.7} and \eqref{eq6.9}, we obtain
 $$|\Upsilon^-_{\lambda_\ast}-\Upsilon^-_{\lambda_n,d_{-}}|\leq |\mathcal{I}^-_{\lambda_\ast}(w_n)-\Upsilon^-_{\lambda_n,d_{-}}|\leq |\mathcal{I}^-_{\lambda_\ast}(w_n)-\mathcal{I}^-_{\lambda_n}(w_n)|+|\mathcal{I}^-_{\lambda_n}(w_n)-\Upsilon^-_{\lambda_n,d_{-}}|\to 0~\text{as}~n\to\infty.$$
 It follows that $\Upsilon^-_{\lambda_n,d_{-}}\to \Upsilon^-_{\lambda_\ast}~\text{as}~n\to\infty$, which is again a contradiction. Thus we must have $\displaystyle\lim_{\lambda\downarrow\lambda_\ast} \Upsilon^-_{\lambda,d_{-}}=\Upsilon^-_{\lambda_\ast}$.\\
 Similarly, we can prove that the map $\lambda\mapsto  \Upsilon^+_{\lambda,d_{+}} $ is decreasing for all $\lambda\in(\lambda_\ast,\lambda_\ast+\epsilon)$ and $\displaystyle\lim_{\lambda\downarrow\lambda_\ast} \Upsilon^+_{\lambda,d_{+}}=\Upsilon^+_{\lambda_\ast}$. Hence the proof is completed.
\end{proof}
\begin{lemma}\label{lem6.6}
Let us choose $\mathrm{d}_{-}\in(0,\mathrm{d}_{\lambda_\ast,-})$ and $\hat{C}^-_{2,\lambda_\ast}<\hat{C}^-_{2}$, then there exists $\varepsilon^->0$ such that   for all $\lambda\in(\lambda_\ast,\lambda_\ast+\varepsilon^-)$,  $\Upsilon^-_{\lambda,d_{-}}$ has a minimizer  $\overline{w}_\lambda\in\mathcal{M}^-_{\lambda_\ast,d_{-}}$.    
\end{lemma}
\begin{proof}
   Let $\varepsilon^->0$ such that for all $\lambda\in(\lambda_\ast,\lambda_\ast+\varepsilon^-)$, there exists a minimizing sequence $\{\overline{w}_n(\lambda)\}_{n\in\mathbb{N}}\subset\mathcal{M}^-_{\lambda_\ast,d_{-}}$ for $\Upsilon^-_{\lambda,d_{-}}$, i.e., we have 
   \begin{equation}\label{eq6.10}
       \lim_{n\to\infty} \mathcal{I}^-_\lambda(\overline{w}_n(\lambda))=\Upsilon^-_{\lambda,d_{-}}.
   \end{equation}
   It follows from the definition of $\mathcal{M}^-_{\lambda_\ast,d_{-}}$ that $\{\overline{w}_n(\lambda)\}_{n\in\mathbb{N}}$ is bounded and hence up to a subsequence $\overline{w}_n(\lambda)\rightharpoonup\overline{w}(\lambda)$ weakly in $W^{s,p}_{V}(\mathbb{R}^N)$ as $n\to\infty$. Since $\overline{w}_n(\lambda)\in\mathcal{M}^-_{\lambda_\ast}$, therefore as done in Lemma \ref{lem4.2}, we can obtain $\overline{w}(\lambda)\neq 0$. Next, we claim that $\overline{w}(\lambda)\in\mathcal{M}^-_{\lambda_\ast,d_{-}} $ for all $\lambda\in(\lambda_\ast,\lambda_\ast+\varepsilon^-)$. Indeed, if not, then let us choose a sequence $\overline{w}(\lambda_m)\notin\mathcal{M}^-_{\lambda_\ast,d_{-}}$ for $\lambda_m\downarrow\lambda_\ast$ as $m\to\infty$, that is, $\lambda_m\in(\lambda_\ast,\lambda_\ast+\varepsilon^-)$ for $m$ large enough. Therefore, from \eqref{eq6.10}, we have
   \begin{equation}\label{eq6.11}
       |\Upsilon^-_{\lambda_m,d_{-}}-\mathcal{I}^-_{\lambda_m}(\overline{w}_n(\lambda_m))|\to 0 ~\text{as}~n,m\to\infty.
   \end{equation}
   Now from Proposition \ref{prop6.5} and \eqref{eq6.11}, we obtain 
   \begin{equation}\label{eq6.12}
       |\Upsilon^-_{\lambda_\ast}-\mathcal{I}^-_{\lambda_m}(\overline{w}_n(\lambda_m))|\to 0 ~\text{as}~n,m\to\infty.
   \end{equation}
   Define $$w_n(\lambda)=t^-_{\lambda}(\overline{w}_n(\lambda))\overline{w}_n(\lambda)~\text{and}~w_{n,m}=w_n(\lambda_m)=t^-_{\lambda_m}(\overline{w}_n(\lambda_m))\overline{w}_n(\lambda_m).$$
   By Corollary \ref{cor6.2}, we have that $w_n(\lambda)\in\mathcal{M}^-_\lambda$ for all $\lambda\in(\lambda_\ast,\lambda_\ast+\epsilon)$ and hence $w_{n,m}\in \mathcal{M}^-_{\lambda_m}$ for $m$ large enough. The boundedness of $\overline{w}_n(\lambda_m)$, Lemma \ref{lem3.2}-$(ii)$ and Lemma \ref{lem3.7}-$(i)$ infer that there exists $a>0$ such that $a<t^-_{\lambda_m}(\overline{w}_n(\lambda_m))<t^-_{\lambda_\ast}(\overline{w}_n(\lambda_m))=1$ (since $\overline{w}_n(\lambda_m)\in\mathcal{M}^-_{\lambda_\ast}$ for $m$ large enough and for all $n\in \mathbb{N}$) for $n,m$ large enough. It follows that $\{w_{n,m}\}_{(n,m)\in\mathbb{N}^2}$ is bounded and hence up to a subsequence $w_{n,m}\rightharpoonup w\neq 0$ weakly in $W^{s,p}_{V}(\mathbb{R}^N)$ as $n,m\to\infty$. Our aim is to prove $w_{n,m}\to w\neq 0$ in $W^{s,p}_{V}(\mathbb{R}^N)$ as $n,m\to\infty$. Indeed, if not, then we have  
   $$\liminf_{n,m\to\infty}\mathcal{F}'_{\lambda_m,w_{n,m}}(t_{\lambda_\ast}(w))> \mathcal{F}'_{\lambda_\ast,w}(t_{\lambda_\ast}(w))=0. $$
 It follows that $\mathcal{F}'_{\lambda_m,w_{n,m}}(t_{\lambda_\ast}(w))>0$ for $n,m$ large enough. Since $w_{n,m}\in\mathcal{M}^-_{\lambda_m}$ for $m$ large enough, therefore by Proposition \ref{prop3.1}, the map $\mathcal{F}_{\lambda_m,w_{n,m}}$ is strictly increasing in $(t^+_{\lambda_m}(w_{n,m}), t^-_{\lambda_m}(w_{n,m})=1)$ and hence $t^+_{\lambda_m}(w_{n,m})<t_{\lambda_\ast}(w)<t^-_{\lambda_m}(w_{n,m})=1$ for $n,m$ large enough. Now from \eqref{eq6.12}, we have 
 $$ \Psi_{\lambda_\ast}(t_{\lambda_\ast}(w)w)<\liminf_{n,m\to\infty}\Psi_{\lambda_m}(t_{\lambda_\ast}(w)w_{n,m})=\liminf_{n,m\to\infty}\mathcal{F}_{\lambda_m,w_{n,m}}(t_{\lambda_\ast}(w))$$ $$\hspace{2.5cm}< \liminf_{n,m\to\infty}\Psi_{\lambda_m}(t^-_{\lambda_m}(w_{n,m})w_{n,m})=\liminf_{n,m\to\infty}\mathcal{I}^-_{\lambda_m}(w_{n,m})=\Upsilon^-_{\lambda_\ast},$$
 which is an absurd because if we take $\lambda'_n\uparrow\lambda_\ast$ as $n\to\infty$, then by Corollary \ref{cor5.3} and Proposition \ref{prop5.5}, we get
 $$\Upsilon^-_{\lambda_\ast}=\lim_{\lambda'_n\uparrow{\lambda_\ast}}\Upsilon^-_{\lambda'_n}\leq \lim_{\lambda'_n\uparrow{\lambda_\ast}}\Psi_{\lambda'_n}(t^-_{\lambda'_n}(w)w)=\Psi_{\lambda_\ast}(t_{\lambda_\ast}(w)w). $$
 Therefore, we must have that $w_{n,m}\to w$ in $W^{s,p}_{V}(\mathbb{R}^N)$ and also up to a subsequence  $t^-_{\lambda_m}(\overline{w}_n(\lambda_m))\to t\geq a$ as $n,m\to\infty$ (since $t^-_{\lambda_m}(\overline{w}_n(\lambda_m))$ is bounded). This infers that $\overline{w}_n(\lambda_m)\to\overline{w}\neq 0 $ in $W^{s,p}_{V}(\mathbb{R}^N)$ as $n,m\to\infty$. Hence for $m$ large enough, we have
 $$0\leq \|\overline{w}(\lambda_m)-\overline{w}\|\leq\liminf_{n\to\infty}\|\overline{w}_n(\lambda_m)-\overline{w}\|. $$
 Passing limit $m\to\infty$ in the above inequality, we get 
 $\overline{w}(\lambda_m)\to \overline{w} $ in $W^{s,p}_{V}(\mathbb{R}^N)$ as $m\to\infty$. Furthermore, from \eqref{eq6.12} and strong convergence, we get $\Upsilon^-_{\lambda_\ast}=\mathcal{I}^-_{\lambda_\ast}(\overline{w})$ and $\overline{w}\in \mathcal{M}^-_{\lambda_\ast}\cup\mathcal{M}^0_{\lambda_\ast}$. But from the definition of $\mathcal{M}^-_{\lambda_\ast,d_{-}}$, we have $\mathrm{dist}(\overline{w}_n(\lambda_m),\mathcal{M}^0_{\lambda_\ast})>d_{-}>0$  for $m$ large enough. It follows that $$\lim_{n,m\to\infty} \mathrm{dist}(\overline{w}_n(\lambda_m),\mathcal{M}^0_{\lambda_\ast})=\mathrm{dist}(\overline{w},\mathcal{M}^0_{\lambda_\ast})\geq d_{-}>0.$$
 This yields $\overline{w}\notin\mathcal{M}^0_{\lambda_\ast}$ and hence $\overline{w}\in\mathcal{M}^-_{\lambda_\ast}$. Thus we have $\overline{w}\in\mathcal{S}^-_{\lambda_\ast}$ and $\overline{w}(\lambda_m)\in\mathcal{M}^-_{\lambda_\ast,d_{-}}$ for $m$ large enough, which is again a contradiction. Therefore, there exists $\varepsilon^->0$ such that $\overline{w}(\lambda)\in\mathcal{M}^-_{\lambda_\ast,d_{-}} $ for all $\lambda\in(\lambda_\ast,\lambda_\ast+\varepsilon^-)$. Arguing as before we conclude that $\overline{w}_n(\lambda)\to\overline{w}(\lambda)$ in $W^{s,p}_{V}(\mathbb{R}^N)$ as $n\to\infty$ and $\mathcal{I}^-_\lambda(\overline{w}(\lambda))=\Upsilon^-_{\lambda,d_{-}}$. By choosing $\overline{w}_\lambda=\overline{w}(\lambda)$, the proof is completed.
\end{proof}
\begin{lemma}\label{lem6.7}
Let us choose $\mathrm{d}_{+}\in(0,\mathrm{d}_{\lambda_\ast,+})$ and $\hat{C}^+_{1}<\hat{C}^+_{1,\lambda_\ast}$, then there exists $\varepsilon^+>0$ such that   for all $\lambda\in(\lambda_\ast,\lambda_\ast+\varepsilon^+)$, $\Upsilon^+_{\lambda,d_{+}}$ has a minimizer  $\overline{u}_\lambda\in\mathcal{M}^+_{\lambda_\ast,d_{+}}$.    
\end{lemma}
\begin{proof}
    The proof is similar to Lemma \ref{lem6.6}.
\end{proof}
\begin{theorem}\label{thm6.8}
   The problem \eqref{main problem} has at least two solutions $w_\lambda\in\mathcal{M}^-_\lambda$ and $u_\lambda\in\mathcal{M}^+_\lambda$ when $\lambda\in(\lambda_\ast,\lambda_\ast+\epsilon)$, where $\epsilon>0$ is small enough.
\end{theorem}
\begin{proof}
 Due to Lemma \ref{lem6.6}, we have  $\overline{w}_\lambda\in\mathcal{M}^-_{\lambda_\ast,d_{-}}$  is a minimizer for $\Upsilon^-_{\lambda,d_{-}}$ when $\lambda\in(\lambda_\ast,\lambda_\ast+\varepsilon^-)$. Consequently, by Corollary \ref{cor6.2} we obtain $t^-_\lambda(\overline{w}_\lambda)\overline{w}_\lambda\in \mathcal{M}^-_{\lambda}$ for $\lambda\in(\lambda_\ast,\lambda_\ast+\epsilon)$. Denote $w_\lambda=t^-_\lambda(\overline{w}_\lambda)\overline{w}_\lambda$, then our aim is to show that $w_\lambda\in \mathcal{M}^-_{\lambda}$ is a solution for \eqref{main problem}. For this, we first have to prove $\overline{w}_\lambda$ is an interior point of $\mathcal{M}^-_{\lambda_\ast,d_{-}}$, i.e., to prove $$\|\overline{w}_\lambda\|<\hat{C}^-_2 ~\text{ for all}~ \lambda\in(\lambda_\ast,\lambda_\ast+\varepsilon^-),$$ where the constant $\hat{C}^-_2 >0$ satisfies $0<C_{\lambda_\ast}<\hat{C}^-_2 $, $C_{\lambda_\ast}$ as in Corollary \ref{cor5.7}. Next, consider a sequence $\{\lambda_n\}_{n\in\mathbb{N}}$ such that $\lambda_n\downarrow\lambda_\ast$ as $n\to\infty$. It follows that $\overline{w}_{\lambda_n}\in\mathcal{M}^-_{\lambda_\ast,d_{-}}$  for $n$ large enough and $\{\overline{w}_{\lambda_n}\}_{n\in\mathbb{N}}$ is bounded. Therefore, without loss of generality, we can assume $\overline{w}_{\lambda_n}\rightharpoonup \overline{w}$ weakly in $W^{s,p}_{V}(\mathbb{R}^N)$ as $n\to\infty$. It follows from Lemma \ref{lemma2.3} that $\overline{w}_{\lambda_n}\to \overline{w}$ in $L^\tau(\mathbb{R}^N)$ and $L^\gamma(\mathbb{R}^N)$ respectively as $n\to\infty$, $\overline{w}_{\lambda_n}\to \overline{w}$ a.e. in $\mathbb{R}^N$ and $\overline{w}\geq 0$. Arguing a similar strategy to that in Lemma \ref{lem4.2}, we obtain $\overline{w} \neq 0$ and $\displaystyle\int_{\mathbb{R}^N}\beta(x)|\overline{w}|^{\gamma}~\mathrm{d}x>0$. We claim that $\overline{w}_{\lambda_n}\to \overline{w}$ in $W^{s,p}_{V}(\mathbb{R}^N)$ as $n\to\infty$. Indeed, if not, then we have  $$\liminf_{n\to\infty}\mathcal{F}'_{\lambda_n,\overline{w}_{\lambda_n}}(t_{\lambda_\ast}(\overline{w}))> \mathcal{F}'_{\lambda_\ast,\overline{w}}(t_{\lambda_\ast}(\overline{w}))=0.$$ 
 It follows that $\mathcal{F}'_{\lambda_n,\overline{w}_{\lambda_n}}(t_{\lambda_\ast}(\overline{w}))>0$ for $n$ large enough. By Corollary \ref{cor6.2}, we have $t^-_{\lambda_n}(\overline{w}_{\lambda_n})\overline{w}_{\lambda_n}\in\mathcal{M}^-_{\lambda_n}$ for $n$ large enough, therefore by Proposition \ref{prop3.1}, the map $\mathcal{F}_{\lambda_n,\overline{w}_{\lambda_n}}$ is strictly increasing in $(t^+_{\lambda_n}(\overline{w}_{\lambda_n}),t^-_{\lambda_n}(\overline{w}_{\lambda_n}))$ and therefore $t^+_{\lambda_n}(\overline{w}_{\lambda_n})<t_{\lambda_\ast}(\overline{w})<t^-_{\lambda_n}(\overline{w}_{\lambda_n})$ for $n$ large enough. Therefore, by Proposition \ref{prop6.5}, we have 
 $$ \Psi_{\lambda_\ast}(t_{\lambda_\ast}(\overline{w})\overline{w})<\liminf_{n\to\infty}\Psi_{\lambda_n}(t_{\lambda_\ast}(\overline{w})\overline{w}_{\lambda_n})=\liminf_{n\to\infty}\mathcal{F}_{\lambda_n,\overline{w}_{\lambda_n}}(t_{\lambda_\ast}(w))$$ $$ < \liminf_{n\to\infty}\Psi_{\lambda_n}(t^-_{\lambda_n}(\overline{w}_{\lambda_n})\overline{w}_{\lambda_n})=\liminf_{n\to\infty}\mathcal{I}^-_{\lambda_n}(\overline{w}_{\lambda_n})=\liminf_{n\to\infty}\Upsilon^-_{\lambda_n,d_{-}}=\Upsilon^-_{\lambda_\ast},$$
 which is absurd because if we take $\lambda'_n\uparrow\lambda_\ast$ as $n\to\infty$, then by Corollary \ref{cor5.3} and Proposition \ref{prop5.5}, we get
 $$\Upsilon^-_{\lambda_\ast}=\lim_{\lambda'_n\uparrow{\lambda_\ast}}\Upsilon^-_{\lambda'_n}\leq \lim_{\lambda'_n\uparrow{\lambda_\ast}}\Psi_{\lambda'_n}(t^-_{\lambda'_n}(\overline{w})\overline{w})=\Psi_{\lambda_\ast}(t_{\lambda_\ast}(\overline{w})\overline{w}). $$
 Thus we conclude that  $\overline{w}_{\lambda_n}\to \overline{w}$ in $W^{s,p}_{V}(\mathbb{R}^N)$ as $n\to\infty$ and hence  $\mathcal{F}'_{\lambda_\ast,\overline{w}}(1)=0$ and $\mathcal{F}''_{\lambda_\ast,\overline{w}}(1)\leq 0$ (since $\overline{w}_{\lambda_n}\in\mathcal{M}^-_{\lambda_\ast}$ for $n$ large enough). It follows that $\overline{w}\in \mathcal{M}^-_{\lambda_\ast}\cup\mathcal{M}^0_{\lambda_\ast}$. But we have $$\lim_{n\to\infty} \mathrm{dist}(\overline{w}_n(\lambda_n),\mathcal{M}^0_{\lambda_\ast})=\mathrm{dist}(\overline{w},\mathcal{M}^0_{\lambda_\ast})\geq d_{-}>0 $$ and hence $\overline{w}\notin \mathcal{M}^0_{\lambda_\ast}$. This shows that we must have $\overline{w} \in \mathcal{M}^-_{\lambda_\ast}$. Next, it is easy to see that $t^-_{\lambda_n}(\overline{w}_{\lambda_n})$ is bounded and therefore, up to a subsequence we can assume $t^-_{\lambda_n}(\overline{w}_{\lambda_n})\to t$ as $n\to\infty$. Using the strong convergence and the fact that $t^-_{\lambda_n}(\overline{w}_{\lambda_n})\overline{w}_{\lambda_n}\in\mathcal{M}^-_{\lambda_n}$ for $n$ large enough, we obtain  $\mathcal{F}'_{\lambda_\ast,\overline{w}}(t)=0$ and $\mathcal{F}''_{\lambda_\ast,\overline{w}}(t)\leq 0$. Hence, from $\overline{w} \in \mathcal{M}^-_{\lambda_\ast}$ it follows that $t=1$. Now from Proposition \ref{prop6.5}, we get
 $$\mathcal{I}^-_{\lambda_\ast}(\overline{w})=\Psi_{\lambda_\ast}(\overline{w})=\lim_{n\to\infty} \Psi_{\lambda_n}(t^-_{\lambda_n}(\overline{w}_{\lambda_n})\overline{w}_{\lambda_n})=\lim_{n\to\infty}\mathcal{I}^-_{\lambda_n}(\overline{w}_{\lambda_n})=\lim_{n\to\infty}\Upsilon^-_{\lambda_n,d_{-}}=\Upsilon^-_{\lambda_\ast}.$$
 This shows that $\overline{w}\in\mathcal{S}^-_{\lambda_\ast}$ and hence from strong convergence, we obtain
 $$\|\overline{w}\|=\lim_{n\to\infty} \|\overline{w}_{\lambda_n}\|\leq C_{\lambda_\ast}<\hat{C}^-_2  ~(\text{see Corollary \ref{cor5.7}}). $$
 From the above inequality, we infer that $\lim_{\lambda\downarrow\lambda_\ast}\|\overline{w}_{\lambda}\|<\hat{C}^-_2 $, that is, $\|\overline{w}_\lambda\|<\hat{C}^-_2 ~\text{ for all}~ \lambda\in(\lambda_\ast,\lambda_\ast+\varepsilon^-)$.
 
 To prove $w_\lambda\in \mathcal{M}^-_{\lambda}$ is a solution for \eqref{main problem}, let us choose $v\in\mathcal{K}$. It follows from $\overline{w}_\lambda\in\mathcal{M}^-_{\lambda_\ast,d_{-}}$ that $\overline{w}_\lambda\in\mathcal{M}^-_{\lambda_\ast}$ for all $\lambda\in(\lambda_\ast,\lambda_\ast+\varepsilon^-)$. Now let us consider that there exists a $\boldsymbol{\sigma}_0>0$ small enough with $0<\boldsymbol{\sigma}<\boldsymbol{\sigma}_0$ satisfying $t^-_{\lambda_\ast}(\overline{w}_\lambda+\boldsymbol{\sigma}v)(\overline{w}_\lambda+\boldsymbol{\sigma}v)\in \mathcal{M}^-_{\lambda_\ast}$. Therefore, by applying the implicit function theorem as in Lemma \eqref{lem4.4}-$(ii)$, we can easily deduce that $t^-_{\lambda_\ast}(\overline{w}_\lambda+\boldsymbol{\sigma}v)\to 1$ as $\boldsymbol{\sigma}\to 0^+$. In addition, we have $\|\overline{w}_\lambda\|<\hat{C}^-_2$ and $\mathrm{dist}(\overline{w}_\lambda,\mathcal{M}^0_{\lambda_\ast})>d_{-}$. This yields
 $$\|t^-_{\lambda_\ast}(\overline{w}_\lambda+\boldsymbol{\sigma}v)(\overline{w}_\lambda+\boldsymbol{\sigma}v)\|< \hat{C}^-_2 $$
 and $$\mathrm{dist}(t^-_{\lambda_\ast}(\overline{w}_\lambda+\boldsymbol{\sigma}v)(\overline{w}_\lambda+\boldsymbol{\sigma}v),\mathcal{M}^0_{\lambda_\ast})>d_{-}~\text{for}~\boldsymbol{\sigma}>0~\text{small enough}. $$
  Hence 
  \begin{equation}\label{eq6.15}
   t^-_{\lambda_\ast}(\overline{w}_\lambda+\boldsymbol{\sigma}v)(\overline{w}_\lambda+\boldsymbol{\sigma}v)\in\mathcal{M}^-_{\lambda_\ast,d_{-}}~\text{for}~\boldsymbol{\sigma}>0~\text{small enough}.   
  \end{equation}
 Consequently, by using Corollary \ref{cor6.2}, we have
 \begin{equation}\label{eq6.16}
  t^-_\lambda\big(t^-_{\lambda_\ast}(\overline{w}_\lambda+\boldsymbol{\sigma}v)(\overline{w}_\lambda+\boldsymbol{\sigma}v)\big)\big(t^-_{\lambda_\ast}(\overline{w}_\lambda+\boldsymbol{\sigma}v)(\overline{w}_\lambda+\boldsymbol{\sigma}v)\big)\in\mathcal{M}^-_\lambda    
 \end{equation}
 for $\boldsymbol{\sigma}>0$ small enough and $\lambda>\lambda_\ast$ lies in an appropriate range. Now, from \eqref{eq6.15} and Lemma \ref{lem6.6}, we get
 \begin{equation}\label{eq6.17}
     \begin{split}
         \Psi_\lambda\bigg(t^-_\lambda\big(t^-_{\lambda_\ast}(\overline{w}_\lambda+\boldsymbol{\sigma}v)(\overline{w}_\lambda+\boldsymbol{\sigma}v)\big)\big(t^-_{\lambda_\ast}(\overline{w}_\lambda+\boldsymbol{\sigma}v)(\overline{w}_\lambda+\boldsymbol{\sigma}v)\big)\bigg)\\& \hspace{-10cm}=\mathcal{I}^-_\lambda\big(t^-_{\lambda_\ast}(\overline{w}_\lambda+\boldsymbol{\sigma}v)(\overline{w}_\lambda+\boldsymbol{\sigma}v)\big)\geq \Upsilon^-_{\lambda,d_{-}}=\mathcal{I}^-_{\lambda}(\overline{w}_\lambda)=\Psi_\lambda(t^-_\lambda(\overline{w}_\lambda)\overline{w}_\lambda).
     \end{split}
 \end{equation}
 But for $\boldsymbol{\sigma}>0$ very small enough, we have 
 \begin{equation}\label{eq6.18}
   \Psi_\lambda\bigg(t^-_\lambda\big(t^-_{\lambda_\ast}(\overline{w}_\lambda+\boldsymbol{\sigma}v)(\overline{w}_\lambda+\boldsymbol{\sigma}v)\big)\big(t^-_{\lambda_\ast}(\overline{w}_\lambda+\boldsymbol{\sigma}v)\overline{w}_\lambda\big)\bigg)\simeq \Psi_\lambda(t^-_\lambda(\overline{w}_\lambda)\overline{w}_\lambda).  
 \end{equation}
 From \eqref{eq6.16} and using the fact that $t^-_\lambda(\overline{w}_\lambda)\overline{w}_\lambda\in \mathcal{M}^-_{\lambda}$, we can apply the implicit function theorem to the same function $\mathcal{H}$ as in Lemma \ref{lem4.4}-$(ii)$ at the point 
$$\Bigg(t^-_\lambda(\overline{w}_\lambda),M(\|\overline{w}_\lambda\|^p)\|\overline{w}_\lambda\|^p,\int_{\mathbb{R}^{N}}\alpha(x)|\overline{w}_\lambda|^{1-\delta}~\mathrm{d}x,\int_{\mathbb{R}^{N}}\beta (x)|\overline{w}_\lambda|^\gamma~\mathrm{d}x\Bigg) $$
and easily verify 
$$t^-_\lambda\big(t^-_{\lambda_\ast}(\overline{w}_\lambda+\boldsymbol{\sigma}v)(\overline{w}_\lambda+\boldsymbol{\sigma}v)\big)\to t^-_\lambda(\overline{w}_\lambda)~\text{as}~\boldsymbol{\sigma}\to 0^+.$$
Denote $$\boldsymbol{\vartheta}=t^-_{\lambda_\ast}(\overline{w}_\lambda+\boldsymbol{\sigma}v)(\overline{w}_\lambda+\boldsymbol{\sigma}v).$$
Now from \eqref{eq6.17} and \eqref{eq6.18}, we obtain 
$$\frac {\Psi_\lambda\big(t^-_\lambda(\boldsymbol{\vartheta}) t^-_{\lambda_\ast}(\overline{w}_\lambda+\boldsymbol{\sigma}v)(\overline{w}_\lambda+\boldsymbol{\sigma}v)\big)-\Psi_\lambda\big(t^-_\lambda(\boldsymbol{\vartheta}) t^-_{\lambda_\ast}(\overline{w}_\lambda+\boldsymbol{\sigma}v)\overline{w}_\lambda\big)}{\boldsymbol{\sigma}}\geq 0.$$
On simplifying the above inequality, we get
$$\hspace{-3cm}\big(t^-_\lambda(\boldsymbol{\vartheta}) t^-_{\lambda_\ast}(\overline{w}_\lambda+\boldsymbol{\sigma}v)\big)^{p(m+1)}\Bigg[\frac{\widehat{M}(\|\overline{w}_\lambda+\boldsymbol{\sigma}v\|^p)-\widehat{M}(\|\overline{w}_\lambda\|^p)}{p\boldsymbol{\sigma}}\Bigg]$$ $$\hspace{2cm}\geq \big(t^-_\lambda(\boldsymbol{\vartheta}) t^-_{\lambda_\ast}(\overline{w}_\lambda+\boldsymbol{\sigma}v)\big)^\gamma \int_{\mathbb{R}^N}\beta(x)\bigg(\frac{|\overline{w}_\lambda+\boldsymbol{\sigma}v|^\gamma-|\overline{w}_\lambda|^\gamma}{\gamma \boldsymbol{\sigma}}\bigg)~\mathrm{d}x$$ $$\hspace{4cm}+\lambda \big(t^-_\lambda(\boldsymbol{\vartheta}) t^-_{\lambda_\ast}(\overline{w}_\lambda+\boldsymbol{\sigma}v)\big)^{1-\delta}\int_{\mathbb{R}^N}\alpha(x)\bigg(\frac{|\overline{w}_\lambda+\boldsymbol{\sigma}v|^{1-\delta}-|\overline{w}_\lambda|^{1-\delta}}{(1-\delta) \boldsymbol{\sigma}}\bigg)~\mathrm{d}x.$$
Applying Fatou's Lemma to the above inequality and using the facts that $t^-_\lambda(\boldsymbol{\vartheta})\to t^-_\lambda(\overline{w}_\lambda)$, $t^-_{\lambda_\ast}(\overline{w}_\lambda+\boldsymbol{\sigma}v)\to 1$ as $\boldsymbol{\sigma}\to 0^+$ and $w_\lambda=t^-_\lambda(\overline{w}_\lambda)\overline{w}_\lambda$, we get
\begin{equation}\label{eq6.19}
 \langle{L(w_{\lambda}),v}\rangle-\lambda \int_{\mathbb{R}^{N}}\alpha(x)w_\lambda^{-\delta} v~\mathrm{d}x-\int_{\mathbb{R}^{N}}\beta(x)w_\lambda^{\gamma-1} v~\mathrm{d}x\geq 0,~\forall~v\in\mathcal{K}.
\end{equation}
Choose $v\in W^{s,p}_{V}(\mathbb{R}^N)$ and define $\phi_\epsilon= w_\lambda+\epsilon v$, then for each $\epsilon>0$ be given $\phi^+_\epsilon\in\mathcal{K}$. Now replacing $\phi_\epsilon$ with $v$ in \eqref{eq6.19} and following Theorem \ref{thm4.6}, we can easily show that $w_\lambda\in\mathcal{M}^-_{\lambda}$ is a solution to \eqref{main problem}. Similarly, we can prove that $u_\lambda\in\mathcal{M}^+_{\lambda}$ is a solution to \eqref{main problem} for  $\lambda\in(\lambda_\ast,\lambda_\ast+\varepsilon^+)$. Let $\epsilon=\min~\{\varepsilon^-,\varepsilon^+\}$, then for  $\lambda\in(\lambda_\ast,\lambda_\ast+\epsilon)$, the problem  \eqref{main problem} has at least two solutions $w_\lambda\in\mathcal{M}^-_\lambda$ and $u_\lambda\in\mathcal{M}^+_\lambda$. This completes the proof.
\end{proof}
\noindent We complete the proof of Theorem \ref{thm2.6}. 
\begin{proof}[Proof of Theorem \ref{thm2.6}]
It follows directly from Theorem \ref{thm4.6},  Theorem \ref{thm5.6} and Theorem \ref{thm6.8}.
\end{proof}
\section{Conclusion and further remarks}
We believe that the result of our work is still true for the following degenerate Kirchhoff singular problems as well as for singular problems with certain relevant assumptions. For examples:
\begin{itemize}
    \item [(I)]    
    $$\begin{cases}
 M\big( \|u\|^p\big)\bigg[-\Delta_p u+V(x)u^{p-1}\bigg]=\lambda f(x)u^{-\delta}+g(x)u^{\gamma-1}~~\text{in}~~\mathbb{R}^N,\\
~~~~~~u>0~~\text{in}~~\mathbb{R}^N,~ \displaystyle\int_{\mathbb{R}^N} V(x)u^p~\mathrm{d}x<\infty,~ u\in W^{1,p}(\mathbb{R}^N),
  \end{cases}$$
  where $\|\cdot\|^p=\|\cdot\|^p_{W^{1,p}_V(\mathbb{R}^N)}$, $N\geq2$, $\Delta_p u=\text{div}(|\nabla u|^{p-2}\nabla u)$ is the $p$-Laplacian operator, $\lambda>0$ is a real parameter, $0<\delta<1<p<\gamma<p^\ast-1$, where $p^\ast=\frac{Np}{N-p}$ is the critical Sobolev exponent, $f>0$ a.e. in $\mathbb{R}^N$, $g\in L^\infty(\mathbb{R}^N)$ is sign changing in $\mathbb{R}^N$ with $g^+\neq 0$, $M:\mathbb{R}^+\to\mathbb{R}^+$ is a degenerate Kirchhoff function, $V:\mathbb{R}^N\to \mathbb{R}^+ $ is a continuous potential function. 
  
   \item [(II)] 
    $$\begin{cases}
 M\big( \|u\|^p\big)\bigg[\mathcal{L}_1 (u)+V(x)u^{p-1}\bigg]=\lambda f(x)u^{-\delta}+\frac{\alpha}{\alpha+\beta} h(x)u^{\alpha-1}v^\beta~~\text{in}~~\mathbb{R}^N,\\ M\big( \|u\|^p\big)\bigg[\mathcal{L}_1 (v)+V(x)v^{p-1}\bigg]=\mu g(x)v^{-\delta}+\frac{\beta}{\alpha+\beta} h(x)u^{\alpha}v^{\beta-1}~~\text{in}~~\mathbb{R}^N,\\
~~~~~~u,v>0~~\text{in}~~\mathbb{R}^N,~ \displaystyle\int_{\mathbb{R}^N} V(x)(u^p+v^p)~\mathrm{d}x<\infty,~ u,v\in \mathbf{X},
  \end{cases}$$  
  where $\mathbf{X}=W^{1,p}(\mathbb{R}^N)~(\text{or}~W^{s,p}(\mathbb{R}^N))$ and we also define $\mathbf{Y}=W^{1,p}_V(\mathbb{R}^N)~(\text{or}~W^{s,p}_V(\mathbb{R}^N))$, with $\|\cdot\|^p=\|\cdot\|^p_{\mathbf{Y}}$, $N\geq2$, the operator $\mathcal{L}_1$ is defined by $\mathcal{L}_1(u)=-\Delta_p u=-\text{div}(|\nabla u|^{p-2}\nabla u)~(\text{or}~\big(-\Delta_p\big)^s)$, called the $p$-Laplacian  (or fractional $p$-Laplacian) operator , $\alpha,\beta>1$, $0<\delta<1<p<\alpha+\beta<p^\ast$, where $p^\ast=\frac{Np}{N-p}$ \big(or $\frac{Np}{N-sp}$\big), called the critical Sobolev exponent, the real parameters $\lambda,\mu>0$, $f,g>0$ a.e. in $\mathbb{R}^N$, $h\in L^\infty(\mathbb{R}^N)$ is sign changing in $\mathbb{R}^N$ with $h^+\neq 0$, $M:\mathbb{R}^+\to\mathbb{R}^+$ is a degenerate Kirchhoff function, $V:\mathbb{R}^N\to \mathbb{R}^+ $ is a continuous potential function.
  \item [(III)] 
    
    $$\begin{cases}
 M\big( \|u\|^p\big)\bigg[\mathcal{L}_p (u)+V(x)u^{p-1}\bigg]+ M\big( \|u\|^q\big)\bigg[\mathcal{L}_q (u)+V(x)u^{q-1}\bigg]\\ \hspace{6.5cm} =\lambda f(x)u^{-\delta}+g(x)u^{\gamma-1}~~\text{in}~~\mathbb{R}^N,\\
~u>0~~\text{in}~~\mathbb{R}^N,~ \displaystyle\int_{\mathbb{R}^N} V(x)u^p~\mathrm{d}x<\infty,~\displaystyle \int_{\mathbb{R}^N} V(x)u^q~\mathrm{d}x<\infty, u\in \mathbb{X},
  \end{cases}$$
  
 where $\mathbf{X}=W^{1,p}(\mathbb{R}^N)\cap W^{1,q}(\mathbb{R}^N)~(\text{or}~~W^{s,p}(\mathbb{R}^N)\cap W^{s,q}(\mathbb{R}^N))$ and we define $\mathbf{Y}=W^{1,m}_V(\mathbb{R}^N)$  $(~\text{or}~W^{s,m}_V(\mathbb{R}^N)$, with $\|\cdot\|^m=\|\cdot\|^m_{\mathbf{Y}}$ for $m\in\{p,q\}$, $N\geq2$, $\mathcal{L}_m u=-\Delta_m u$ (or $\big(-\Delta_m\big)^s$ ) for $m=\{p,q\}$, is the $m$-Laplacian (or fractional $m$-Laplacian) operator, $\lambda>0$ is a real parameter, $1<\delta<q<p<\gamma<q^\ast-1$,  where $q^\ast=\frac{Nq}{N-q}$ \big(or $\frac{Np}{N-sp}$\big) is the critical Sobolev exponent, $f>0$ a.e. in $\mathbb{R}^N$, $g\in L^\infty(\mathbb{R}^N)$ is sign changing in $\mathbb{R}^N$ with $g^+\neq 0$, $M:\mathbb{R}^+\to\mathbb{R}^+$ is a degenerate Kirchhoff function, $V:\mathbb{R}^N\to \mathbb{R}^+ $ is a continuous  potential function. 
  \item [(IV)] 
    \begin{equation}\label{eq7.6}
    \begin{cases}
\mathfrak{L }^{p,q}_{a,V}(u) =\lambda f(x)u^{-\delta}+g(x)u^{\gamma-1}~~\text{in}~~\mathbb{R}^N,\\
~u>0~~\text{in}~~\mathbb{R}^N,~ \displaystyle\int_{\mathbb{R}^N} V(x)
(u^p+a(x)u^q)~\mathrm{d}x<\infty, u\in W^{1,\mathcal{H}}(\mathbb{R}^N),
  \end{cases}\tag{$\mathcal{P}_\lambda$}
  \end{equation}
  where $N\geq2$, $V:\mathbb{R}^N\to \mathbb{R}^+ $ is a positive continuous  function and $\lambda>0$ is a parameter. The double phase type operator that appears in the problem \eqref{eq7.6} is defined as follows:
  $$\mathfrak{L }^{p,q}_{a,V}(u)=-\text{div}(|\nabla u|^{p-2}\nabla u+a(x)|\nabla u|^{q-2}\nabla u)+V(x)(|u|^{p-2}u+a(x)|u|^{q-2}u). $$
  Also, we have $0\leq a(\cdot)\in L^1(\mathbb{R}^N)$, $0<\delta<1<p<q<N,~\frac{p}{q}<1+\frac{1}{N},~q<p^\ast$, where $p^\ast=\frac{Np}{N-p}$ is the critical Sobolev exponent,  $f>0$ a.e. in $\mathbb{R}^N$, $g\in L^\infty(\mathbb{R}^N)$ is sign changing in $\mathbb{R}^N$ with $g^+\neq 0$.
\end{itemize}
\begin{remark}
    The above problems are still actual whenever $M\equiv 1$. Moreover, one can study such singular problems in the bounded domain. We will leave the details proof of the above issues to interested readers.
\end{remark}
\section*{Acknowledgements} DKM wants to sincerely thank DST INSPIRE Fellowship DST/INSPIRE/03/2019/000265 sponsored by Govt. of India. TM acknowledges the support of the Start-up Research Grant from DST-SERB, sanction no. SRG/2022/000524. AS was supported by the DST-INSPIRE Grant DST/INSPIRE/04/2018/002208 sponsored by Govt. of India.

\bibliography{ref}
\bibliographystyle{abbrv}
\Addresses
\end{document}